\documentclass[a4paper,11pt,leqno]{article}
\usepackage{geometry}
\usepackage{tikz,tikz-cd,amsmath,amsfonts,amssymb,graphicx}
\usetikzlibrary{matrix,arrows,decorations.pathmorphing}
\usepackage[latin1]{inputenc}
\usepackage{amsthm}
\usepackage[autostyle]{csquotes} 
\usepackage{enumerate}
\usepackage[mathscr]{euscript}
\usepackage{mathtools}
\usepackage{hyperref}
\usepackage{url}
\usepackage{scalerel}
\usepackage{dsfont}
\usepackage{stackengine,wasysym}

\usepackage[noabbrev,nameinlink,capitalize]{cleveref} 

\numberwithin{equation}{section}
\theoremstyle{definition}

\newtheorem{theorem}[equation]{Theorem}
\newtheorem{corollary}[equation]{Corollary}
\newtheorem{definition}[equation]{Definition}
\newtheorem{lemma}[equation]{Lemma}
\newtheorem{example}[equation]{Example}
\newtheorem{proposition}[equation]{Proposition}
\newtheorem{remark}[equation]{Remark}

\newcommand\reallywidetilde[1]{\ThisStyle{%
		\setbox0=\hbox{$\SavedStyle#1$}%
		\stackengine{-.1\LMpt}{$\SavedStyle#1$}{%
			\stretchto{\scaleto{\SavedStyle\mkern.2mu\AC}{.5150\wd0}}{.6\ht0}%
		}{O}{c}{F}{T}{S}%
}}

\makeatletter
\newcommand{\colim@}[2]{%
	\vtop{\m@th\ialign{##\cr
			\hfil$#1\Operator@font colim$\hfil\cr
			\noalign{\nointerlineskip\kern1.5\ex@}#2\cr
			\noalign{\nointerlineskip\kern-\ex@}\cr}}%
}
\newcommand{\colim}{%
	\mathop{\mathpalette\colim@{\rightarrowfill@\scriptscriptstyle}}\nmlimits@
}
\renewcommand{\varprojlim}{%
	\mathop{\mathpalette\varlim@{\leftarrowfill@\scriptscriptstyle}}\nmlimits@
}
\renewcommand{\varinjlim}{%
	\mathop{\mathpalette\varlim@{\rightarrowfill@\scriptscriptstyle}}\nmlimits@
}
\makeatother
\newcommand{\cate}[1]{\mathscr{#1}}
\newcommand{\s}[1]{\mathbb{S}_{\scalebox{1}{$\scriptscriptstyle #1$}}}

\newcommand{\Sh}[2]{\cate{S}\text{hv}({#1}; {#2})}

\newcommand{\coSh}[2]{\C\text{o}\cate{S}\text{hv}({#1}; {#2})}

\newcommand{\sHom}[3]{\text{\underline{Hom}}_{\scalebox{1}{$\scriptscriptstyle #1$}}(#2, #3)}

\newcommand{\Sec}[2]{\Gamma(#1; #2)}
\newcommand{\KSec}[3]{\Gamma_{\scalebox{1}{$\scriptscriptstyle #3$}}(#1; #2)}
\newcommand{\cSec}[2]{\Gamma_{\scalebox{1}{$\scriptscriptstyle c$}}(#1; #2)}

\newcommand{\pf}[1]{#1_{\ast}}
\newcommand{\pb}[1]{#1^{\ast}}
\newcommand{\pfp}[1]{#1_{!}}
\newcommand{\pbp}[1]{#1^{!}}
\newcommand{\pfs}[1]{#1_{\sharp}}
\newcommand{\cSh}[2]{\cate{S}\text{hv}^c({#1}; {#2})}
\newcommand{\fcSh}[2]{\cate{S}\text{hv}^{fc}({#1}; {#2})}
\newcommand{\ccoSh}[2]{\C\text{o}\cate{S}\text{hv}^c({#1}; {#2})}
\newcommand{\fccoSh}[2]{\C\text{o}\cate{S}\text{hv}^{fc}({#1}; {#2})}

\DeclareMathOperator{\Fun}{Fun}

\DeclareMathOperator{\opposite}{op}

\DeclareMathOperator{\exitpath}{Exit}
\DeclareMathOperator{\singularset}{Sing}
\DeclareMathOperator{\depth}{depth}
\DeclareMathOperator{\linking}{Link}
\DeclareMathOperator{\unzipping}{Unzip}

\newcommand{\op}{^{\opposite}}
\newcommand{\C}{\cate{C}}
\newcommand{\D}{\cate{D}}
\newcommand{\E}{\cate{E}}
\newcommand{\X}{\cate{X}}

\newcommand{\Cat}{\cate{C}\mathrm{at}}
\newcommand{\rnum}{\mathbb{R}}
\newcommand{\exit}[2]{\exitpath_{{#1}}({#2})}
\newcommand{\dpt}[1]{\depth({#1})}
\newcommand{\sing}[1]{\singularset({#1})}
\newcommand{\link}[2]{\linking_{{#1}}({#2})}
\newcommand{\unzip}[2]{\unzipping_{{#1}}({#2})}

\providecommand{\sSet}{\text{s}\cate{S}\text{et}}
\providecommand{\infcat}{\Cat_{\scalebox{1}{$\scriptscriptstyle \infty$}}}
\providecommand{\spaces}{\cate{S}}
\providecommand{\spectra}{\cate{S}\text{p}}

\newsavebox{\pullback}
\sbox\pullback{%
	\begin{tikzpicture}%
		\draw (0,0) -- (1ex,0ex);%
		\draw (1ex,0ex) -- (1ex,1ex);%
\end{tikzpicture}}
\newsavebox{\pushout}
\sbox\pushout{%
	\begin{tikzpicture}%
		\draw (0ex,0ex) -- (0ex,1ex);%
		\draw (0ex,1ex) -- (1ex,1ex);%
\end{tikzpicture}}
\tikzset{%
	symbol/.style={%
		draw=none,
		every to/.append style={%
			edge node={node [sloped, allow upside down, auto=false]{$#1$}}}
	}
}

{}

\begin{document}
		\title{Verdier duality on conically smooth stratified spaces}
	\author{Marco Volpe\footnote{University of Regensburg, Universitätsstraße 31, Regensburg. email: marco.volpe@ur.de. The author was supported by the SFB 1085 (Higher Invariants) in Regensburg,
			Germany, funded by the DFG (German Science Foundation).
	}}
	\maketitle
	\begin{abstract}
		In this paper we prove a duality  for constructible sheaves on conically smooth stratified spaces. Here we consider sheaves with values in a stable and bicomplete $\infty$-category equipped with a closed symmetric monoidal structure, and in this setting constructible means locally constant along strata and with dualizable stalks. The crucial point where we need to employ the geometry of conically smooth structures is in showing that Lurie's version of Verdier duality restricts to an equivalence between constructible sheaves and cosheaves: this requires a computation of the exit paths $\infty$-category of a compact stratified space, that we obtain via resolution of singularities.
	\end{abstract}
	\tableofcontents
	\section{Introduction}
	
	 Constructible sheaves are of great interest both in algebraic and differential geometry, as they provide tools to study invariants for singular spaces (such as \textit{intersection cohomology} \cite{beilinson2018faisceaux}) and have relations with D-modules (see \cite{kashiwara1984riemann}). Roughly speaking, constructible sheaves are sheaves on stratified spaces that behave nicely on strata (see \cref{constrdef} for more details). A fundamental feature of constructible sheaves is that, assuming a finiteness condition on the stalks, they carry a duality sometimes referred to as \textit{Verdier duality}. This is an anti-equivalence of the category of constructible sheaves to itself, which is defined by taking an internal-hom into the \textit{dualizing complex}. Verdier duality has many applications: for example, using abstract trace methods it allows one to associate to any constructible sheaf a class in Borel--Moore homology. One of the interests of these classes is that they can be related to Euler characteristics via computations with the six functor formalism (see \cite[Chapter 9]{kashiwara1990sheaves} for a discussion on classical index formulas and their microlocal enhancements).
	 
	 As far the author knows, it was an idea of MacPherson that, in the case when strata are manifolds, the duality should be thought of as a combination of two different equivalences of categories. The first, induced by the construction of \textit{sections with compact support}, was expected to identify constructible sheaves with constructible \textit{cosheaves}. Without constructibility assumptions, this was proven by Lurie in \cite[Theorem 5.5.5.1]{lurie2017higher}. The second maps back contravariantly construtible cosheaves to sheaves, and is obtained using a foreseen combinatorial description of constructible (co)sheaves, similar in spirit to the monodromy for local systems. Following \cite{barwick2018exodromy}, we refer to this combinatorial description as \textit{exodromy} (see \cite[Theorem A.9.3]{lurie2017higher}, \cite[Theorem 1.2.5]{ayala2017local}). The topological exodromy equivalence makes use of a generalization of the homotopy type of a stratified topological space, that keeps track of the stratification. This is known as the category of \textit{exit paths} of a stratified topological space (see \cite[Definition A.6.2]{lurie2017higher}, \cite{treumann2009exit}). In this paper we make use of the language of $\infty$-categories to realize the vision of MacPherson and prove the expected duality result in a very general setting. 

     The first appearance of a proof of Verdier duality following the approach proposed by MacPherson can be found in \cite[Theorem 7.7]{curry2018dualities}. In that work, the author deals with derived categories of constructible (co)sheaves of vector spaces on locally finite regular CW complexes. This approach was later generalized in \cite{aoki2022posets}. It is worth noting that \cite{aoki2022posets} works with spectra-valued functors on posets. In our setting, a stratum can be any smooth manifold. Hence, via exodromy we obtain a duality for spectra-valued functors on $\infty$-categories that are not necessarily posets. Another mention of Verdier duality appears in \cite[Example 1.10.8]{ayala2019stratified}. In that paper, the authors outline a strategy to prove Verdier duality on stratified topological spaces, which is essentially the same as the one we employ in this paper. However, it is important to highlight that some of the main steps in their outline lack a rigorous proof (see \cref{>AMGR} for a more detailed comment).
	 
	 Let us spend a few words to specify more precisely the framework in which we are working.
	 Relying on the previous paper \cite{volpe2021six}, we will be able to deal with sheaves valued in any stable bicomplete $\infty$-category $\C$, equipped with a closed symmetric monoidal structure. The machinery of six functors developed in \cite{volpe2021six} supplies us with a \textit{dualizing sheaf} $\omega_X^{\C}$ for any $\C$ as above and $X$ a locally compact Hausdorff stratified space. More precicely, if $a: X\rightarrow\ast$ is the unique map, $\omega_X^{\C}$ is defined by applying the functor
	 $$\pbp{a}_{\C}:\C\rightarrow\Sh{X}{\C}$$
	 to the monoidal unit of $\C$. Our duality functor will be given by taking an internal-hom into $\omega_X^{\C}$, and denoted by $D_X^{\C}$.
	 
	 Following the nomenclature of \cite{barwick2018exodromy}, we will define a sheaf with values in $\C$ to be \textit{formally constructible}\footnote{Notice that in this paper we will only deal with sheaves which are constructible with respect to a fixed stratification, as opposed to \cite[Chapter 8]{kashiwara1990sheaves}, for example.} if its restriction to each stratum is locally constant, and \textit{constructible} if furthermore all of its stalks are dualizable. Similar definitions can be given for $\C$-valued cosheaves by observing that, up to passing to an opposite category, these are $\C\op$-valued sheaves. 
	 
	 \begin{remark}
	 	The requirement of dualizability for stalks is unavoidable because, when $X$ is the point, $\omega_X^{\C}$ is the monoidal unit of $\C$ and the duality functor coincides with the one coming from the monoidal structure on $\C$. Furthermore, this assumption is highly reasonable. For example, if $\C= D(R)$ and $R$ is a commutative ring (or more generally, modules over any $E_{\infty}$-ring spectrum), it is a well-known result that a complex is dualizable if and only if it is perfect (see for example \cite[Proposition 7.2.4.4]{lurie2017higher} for a proof of a more general statement about $E_1$-ring spectra). Consequently, under our assumptions, we are able to recover the classical setting as a special case.
	 \end{remark} 
	 
	 For the geometric side of the story, we will consider \textit{conically smooth stratified spaces}\footnote{Most part of our strategy to prove Verdier duality works more generally for $C^0$-stratified spaces (see \cite[Definition 2.1.15]{ayala2017local}). However, the proof of \cref{finexit} relies on the existence of blow-ups, which are not available without the presence a conically smooth structure. In a future paper (see \cite{volpe2023finiteness}), we will use some general facts about stratified homotopy types, to show that the exit path $\infty$-category of a compact $C^0$-stratified space is a compact object in $\infcat$.}. These were introduced by Ayala, Francis and Tanaka in \cite{ayala2017local}, and provide a natural extension of $C^{\infty}$-structures in the stratified setting. Notable examples of stratified spaces admitting a conically smooth atlas are \textit{Whitney stratified spaces}, as proven in \cite{nocera2021whitney}. The definition of conically smooth atlases is rather involved, as it relies on a elaborate inductive construction based on the \textit{depth} of a stratification. For convenience of the reader, we recall the definition of depth.
	 
	 \begin{definition}
	 	Let $s : X\rightarrow P$ be a stratified space. Then the depth is defined as $$\dpt{X} = \sup\limits_{x\in X}(\text{dim}_x(X) - \text{dim}_x(X_{s(x)})),$$ where dim denotes the covering dimension and $X_{s(x)}$ is the stratum of $X$ corresponding to $s(x)\in P$.
	 \end{definition}
 
    We suggest the reader has a look at the introductions of \cite{ayala2017local} or \cite{nocera2021whitney} to get an idea of how this works. 
	 
    The main feature of conically smooth structures we will use in this paper is \textit{unzip} construction (see \cite[Definition 7.3.11]{ayala2017local}), that allows one to functorially resolve any conically smooth stratified space into a manifold with corners. We will give a brief explanation of how this works later in the paper, but for now let us only mention that, if $X_k\hookrightarrow X$ is the inclusion of a statum of maximal depth, it consists of a square
	 \begin{equation}\label{blowupstrata}
	 	\begin{tikzcd}
	 		\link{k}{X} \arrow[d, "{\pi_{X}}"] \arrow[r, hook]  \arrow[dr, phantom, "\usebox\pullback" , very near start, color=black] & \unzip{k}{X} \arrow[d] \\
	 		X_k \arrow[r, hook] & X \arrow[ul, phantom, "\usebox\pushout" , very near start, color=black]
	 	\end{tikzcd}
	\end{equation}
	 which is both pushout and pullback, and $\unzip{k}{X}$ is a conically smooth manifold \textit{boundary} given by $\link{k}{X}$ such that both its interior and $\link{k}{X}$ have depth strictly smaller than the one of $X$. An interesting consequence of the existence of pushout/pullback square (\ref{blowupstrata}) is that the notion of conically smooth map is completely determined by the one of smooth maps between manifolds with corners. 
	 
	 We are now ready to state our main result.
	 
	 \begin{theorem}[\cref{verdduality}]
	 	Let $X$ be a conically smooth stratified space, and let $\cSh{X}{\C}$ be the full subcategory of $\Sh{X}{\C}$ spanned by constructible sheaves. Then the functor $$D_X^{\C}: \Sh{X}{\C}\op\rightarrow\Sh{X}{\C}$$
	 	 restricts to an equivalence of $\infty$-categories
	 	$$
	 	\begin{tikzcd}
	 		D_X^{\C}:\cSh{X}{\C}\op\arrow[r, "\simeq"] & \cSh{X}{\C}.
	 	\end{tikzcd}
	 	$$
	 \end{theorem}
	 
	 To conclude this introduction, let us make a short comment on how our proof strategy goes. As mentioned earlier, our first observation is that the functor $D_X^{\C}$ factors through the equivalence 
	 $$\mathbb{D}_{\C} : \Sh{X}{\C}\rightarrow\coSh{X}{\C},$$
	 proven by Lurie in \cite[Theorem 5.5.5.1]{lurie2017higher}.
	 Most of the work then lies in proving that the restriction of $\mathbb{D}_{\C}$ to constructible sheaves factors through constructible cosheaves. We first show in \cref{dualconstr} that $\omega_X^{\C}$ is constructible when $\C= \spectra$ (the $\infty$-category of spectra), and from the techniques developed in \cite{volpe2021six} we deduce immediately that 
	 $$\pbp{a}_{\C} : \C\rightarrow\Sh{X}{\C}$$
	 factors through formally constructible sheaves. As a consequence of this and some properties of constructible sheaves that follow from \textit{homotopy invariance} (see \cref{htpyinv}), one deduces that $\mathbb{D}_{\C}$ maps formally constructible sheaves into formally constructible cosheaves. We stress that being able to work with such a general class of coefficients, which is closed under passing to opposite categories, makes this step extremely formal.
	 
	 The missing piece is then showing that $\mathbb{D}_{\C}$ preserves the property of having dualizable stalks. This is the point where we have to employ the geometry of conically smooth structures. More specifically, we use the unzip constuction and an inductive argument on the depth to prove that any compact stratified space equipped with a conically smooth structure has a finite exit paths $\infty$-category (\cref{finexit}). For simplicity, let us explain how to use \cref{finexit} in the special case $X = C(Z)$ with $Z$ compact, where $C(Z)$ denotes the cone on $Z$. If $x\in X$ is the cone point and $F$ is any constructible sheaf on $X$, there is a fiber sequence
	 \begin{equation}\label{fibseq}
          \KSec{X}{F}{x}\rightarrow F_x \rightarrow\Sec{Z}{F}.
	 \end{equation}
     Here $\KSec{X}{F}{x}$ denotes the sections of $F$ supported at $x$ (i.e. the stalk of the associated cosheaf of compactly supported sections of $F$) and $F_x$ is the stalk of $F$ at $x$. By \cref{finexit} and the exodromy equivalence (which we show holds also for our general class of coefficients in \cref{exodromy}), one deduces that $\Sec{Z}{F}$ is dualizable. Thus, using the fiber sequence above, we see that $F_x$ is dualizable if and only if $\KSec{X}{F}{x}$ is, which proves our claim.
	 
	 \subsection{Linear overview}
	 
	 We now give a linear overview of the results in our paper. 
	 
	 Section 2 is mainly devoted to the proof \cref{finexit}. In the first part we recall the definition of a finite $\infty$-category, and show how these can be described in the model of quasi-categories. None of these results or definitions are new, but we decided to include a few words on the subject since we could not find any reference dealing with it in our preferred fashion. In the second part we recall Lurie's definition of the simplicial set of exit-paths of a stratified topological space. Given a proper stratified fiber bundle $\pi: L\rightarrow X$, we show how one can conveniently compute the exit-paths of the fiberwise cone of $\pi$ in terms of $L$ and $X$. To prove \cref{finexit}, we cover a conically smooth stratified space with the open subset given by the locus of points of depth zero and a tubular neighbourhood of its complement. By induction and \cref{exitcone}, one is then left to show that the exit paths $\infty$-category of the former is finite. This is proven using the unzip construction. Namely, by unzipping the complement, the open subset of points of depth $0$ can be identified with the interior of a compact manifold with corners. 
	 
	 In Section 3 we deal with extending the results of \cite{Haine2020TheHO} to sheaves values in stable bicomplete $\infty$-categories. This is very formal, after \cite{volpe2021six}. As a consequence, we show that the stalk at a point $x$ of a constructible sheaf is the same as sections at any conical chart around $x$ (\cref{stalksconstr}). We also provide a convenient description of the restriction of a construtible sheaf along a stratum (\cref{restrata}). These two results are essential and are used very often in what follows. For example, the first immediately implies the existence of the fiber sequence (\ref{fibseq}). In the second part we will then characterize constructible sheaves by the property of being homotopy invariant (see \cref{charconstr}), and use this to deduce exodromy for conically smooth spaces with general stable bicomplete coefficients (\cref{exodromy}).
	 
	 Section 4 deals with proving our main result. Given a $C^0$-stratified space $X$,  through an inductive argument on the depth we show that $\omega_X^{\C}$ is constructible (\cref{dualconstr}). We first reduce to proving the statement in the case $\C= \spectra$ by employing the techniques developed in \cite{volpe2021six}. Then the only non-trivial part consists in showing that, when $X$ is a cone, the stalk of the dualizing sheaf at the cone-point is a finite spectrum. We then conclude by proving \cref{verdduality}. As explained at the beginning of the introduction, our argument starts by observing that the duality functor factors through Lurie's Verdier duality. The hard part then consists in showing that the latter restricts to an equivalence between constructible (co)sheaves, for which we have to implement all the results obtained previously in the paper.
	 
	 Finally, in the appendix, we show that the shape of any proper and locally contractible $\infty$-topos is a compact object in the $\infty$-category of $\infty$-groupoids.
	    
	 \subsection{Acknowledgements}   
	 I owe my interest and fascination with the theory of conically smooth stratified spaces to my supervisor Denis-Charles Cisinski: the enthusiasm he showed while explaining to me the main ideas of \cite{ayala2018stratified} gave me the motivation to start reading \cite{ayala2017local}, which incidentally led to the birth of this paper and my collaboration with Guglielmo Nocera \cite{nocera2021whitney}.
	 
	 Of course, the second person I would like to thank is Guglielmo himself. The hours of discussions we've had through zoom during early pandemic times, and later in Regensburg while he stayed there for a visit, have helped me to sharpen my intuition on stratified spaces. The importance of our friendship too cannot be underestimated.
	 
	 I'd also like to thank Benedikt Preis for asking me if it was possible to deduce the duality for constructible sheaves from Lurie's covariant Verdier duality; George Raptis for his support, especially during my self-imposed exile in Italy, and for being a constant reminder of the importance of learning also some geometry, every once in a while; Mauro Porta, Jean-Baptiste Teyssier and Peter Haine for sharing earlier drafts of their papers on homotopy invariance and exodromy with me, and for long and stimulating discussions during my stay in Paris and through zoom; Denis Nardin, Andrea Gagna, Edoardo Lanari, David Ayala, Niklas Kipp, Sebastian Wolf, Han-Ung Kufner and Kaif Hilman for answering (sometimes silly) questions related to the subject of this paper, and for constantly providing stimulating math conversations. 
	 
	 
	 This paper is a part of my PhD thesis.
	 	
	\section{Finite exit paths}
	
	This first section contains the main geometric input needed to achieve our goal. Namely, we show the exit-paths $\infty$-category of a compact conically smooth stratified space is finite (\cref{finexit}). 
	
	\subsection{Finite $\infty$-categories}
	
	This short section is devoted to recalling the definition of a finite $\infty$-category. Before going into that, we say a few words about what an $\infty$-category is for us. In this paper, we work in the model of quasicategories. Let $\sSet$ be the category of simplicial sets. Following \cite[Example 7.10.14]{cisinski2019higher}, we define $\infcat$ as the localization of $\sSet$ at the class of Joyal equivalences. The class of Joyal equivalences and fibrations equip $\sSet$ with the structure of a category with weak equivalences and fibrations in the sense of \cite[Definition 7.4.12]{cisinski2019higher}. Hence, by \cite[Theorem 7.5.18]{cisinski2019higher} it follows that any object in $\infcat$ is equivalent to the image through the localization functor $\sSet\rightarrow\infcat$ of a fibrant object in the Joyal model structure. For this reason, objects of $\infcat$ will be called $\infty$-categories.
	
	The first definition we propose is expressed internally to the $\infty$-category $\infcat$ in terms of pushouts, and so in a kind of model-independent fashion. Later we prove that this is actually equivalent to a notion of finiteness that one might expect in the simplicial model. All the results appearing here are not at all original, but we still felt the need to write this section as, in the process of completing the paper, we could not locate a reference dealing with the subject. In what follows, we will denote by $\spaces$ the full subcategory of $\infcat$ spanned by $\infty$-groupoids.
	
	\begin{definition}\label{fincat}
		An $\infty$-category is said to be \textit{finite} if it belongs to the smallest full subcategory of $\infcat$ which contains $\emptyset$, $\Delta^0$, and $\Delta^1$ and is closed under pushouts. An $\infty$-groupoid is said to be finite if it is so as an $\infty$-category. We will denote by $\infcat^{f}$ and $\spaces^{f}$ respectively the full subcategories of $\infcat$ and $\spaces$ spanned by finite objects.
	\end{definition}
	
	\begin{remark}\label{fingpds}
		Recall that the inculsion $\spaces\hookrightarrow\infcat$ admits both a left and a right adjoint, and one may describe the left adjoint on objects by sending and $\infty$-category $\C$ to the localization $\C[\C^{-1}]$. Thus, since $\spaces\hookrightarrow\infcat$ preserves colimits and $\Delta^1[(\Delta^1)^{-1}]\simeq\Delta^0$, one may identify the class of finite $\infty$-groupoids with the objects of the smallest full subcategory of $\spaces$ which contains $\emptyset$ and $\Delta^0$ and it is closed under pushouts. This implies in particular that, for any finite $\infty$-category $\C$, the localization $\C[\C^{-1}]$ is again finite.
	\end{remark}
	
	\begin{lemma}\label{finloc}
		Let $\C$ be a finite $\infty$-category, and let $W$ be a finite subcategory of $\C$. Then the localization $\C[W^{-1}]$ is again finite.
	\end{lemma}
	
	\begin{proof}
		We have a pushout square
		$$
		\begin{tikzcd}
			W\ar[r] \ar[d] & \C \ar[d] \\
			W[W^{-1}] \ar[r] & \C[W^{-1}], \arrow[ul, phantom, "\usebox\pushout" , very near start, color=black]
		\end{tikzcd}
		$$
		in $\infcat$, thus it suffices to show that $W[W^{-1}]$ is finite. This follows immediately by \cref{fingpds}.
	\end{proof}


	
	Recall that a simplicial set is said to be \textit{finite} if it has a finite number of non-degenerate simplices. In the next proposition we reconcile this notion of finiteness with the one in \cref{fincat}. We will need the following lemma, whose proof was explained to us by Sebastian Wolf.
	
	\begin{lemma}\label{basti}
		Let $\C$ be an $\infty$-category, and let $f:K\rightarrow\C$ be any map of simplicial sets, where $K$ is finite. Moreover, suppose that there exists a finite simplicial set $K'$ and a Joyal equivalence $g: K'\rightarrow\C$. Then ther exists a finite simplicial set $L$, a Joyal equivalence $j:L\rightarrow\C$ and a commutative diagram in $\sSet$
		$$
		\begin{tikzcd}
			& \C \\
			K\arrow[ur, "f"] \arrow[r, hook, "k"] & L\arrow[u, "j"']
		\end{tikzcd}
		$$
		where $k$ is a monomorphism.
	\end{lemma}

    \begin{proof}
    	We define inductively a sequence of finite simplicial sets $\{K'_n\}_{n\in\mathbb{N}}$. We set $K_0' = K'$, and we define $K_n'$ via the pushout
    	$$
    	\begin{tikzcd}
    		\coprod\limits_{\Lambda_j^n\rightarrow K'_{n-1}}\Lambda_j^n\arrow[r]\arrow[d, hook] & K'_{n-1}\arrow[d, hook] \\
    		\coprod\limits_{\Lambda_j^n\rightarrow K'_{n-1}}\Delta^n\arrow[r] & K'_n. \arrow[ul, phantom, "\usebox\pushout" , very near start, color=black]
    	\end{tikzcd}
    	$$
    	Now set
    	$$K'_{\infty}\coloneqq\text{colim}(K_0'\hookrightarrow K_1'\hookrightarrow\dots\hookrightarrow K'_n\hookrightarrow\dots).$$
    	Furthermore, since all horns are finite simplicial sets, any map $\Lambda_j^n\rightarrow K'_{\infty}$ factors through some $K_m'$. By construction of the sequence, we get a commutative diagram 
    	$$
    	\begin{tikzcd}
    		\Lambda_j^n\arrow[r]\arrow[d] & K'_m \arrow[r, hook] \arrow[d, hook] & K'_{\infty} \\
    		\Delta^n\arrow[r, dotted] & K'_{m+1}\arrow[ur, hook]
    	\end{tikzcd}
    	$$
    	which implies that $K'_{\infty}$ is an $\infty$-category. Since the class of categorical anodyne extensions is saturated (see \cite[Definition 3.3.3]{cisinski2019higher}), we see that $K'\hookrightarrow K'_{\infty}$ is a categorical anodyne extension, and in particular a Joyal equivalence. Hence, by the assumption that $\C$ is an $\infty$-category, we get a commutative triangle
    	$$
    	\begin{tikzcd}
    		K'\arrow[r, "g"]\arrow[d, hook] & \C \\
    		K'_{\infty}\arrow[ur, dotted, "\phi"']
    	\end{tikzcd}
    	$$
    	where $\phi$ is a Joyal equivalence by the 2-out-of-3 property. Since $K'_{\infty}$ is also an $\infty$-category, $\phi$ admits a quasi-inverse
    	$$\psi:\C\rightarrow K'_{\infty}.$$
    	By the finiteness of $K$, we get that the composition
    	$$\psi f: K\rightarrow \C\rightarrow K'_{\infty}$$
    	factors through some $\delta:K\rightarrow K'_n$. Thus, we get a triangle
    	$$
    	\begin{tikzcd}
    		& \C & \\
    		K\arrow[ur, "f"]\arrow[r, "\delta"] & K'_n\arrow[r, hook] & K'_{\infty}\arrow[ul, "\phi"']
    	\end{tikzcd}
    	$$
    	which commutes up to $J$-homotopy, where $J$ is the interval object for the Joyal model structure as defined in \cite[Definition 3.3.3]{cisinski2019higher}.
    	
    	Let now $L$ be the mapping cylinder of $\delta$. Since $J$ is a finite simplicial set, $L$ must be finite as well. By the usual factorizations obtained via mapping cylinders, we get a triangle
    	$$
    	\begin{tikzcd}
    		& \C \\
    		K\arrow[ur, "f"] \arrow[r, hook, "i"] & L\arrow[u, "p"']
    	\end{tikzcd}
    	$$
    	commuting up to $J$-homotopy, where $i$ is a monomorphism and $p$ is a Joyal equivalence. If $H$ is a $J$-homotopy between $f$ and $pi$, we may find a map $\tilde{H}$ fitting in the diagram
    	$$
    	\begin{tikzcd}
    		K\times J \cup L\times\{1\}\arrow[r, "{H\cup p}"]\arrow[d, hook] & \C \\
    		L\times J \arrow[ur, "\tilde{H}"', dotted] &
    	\end{tikzcd}
    	$$
    	since $K\times J \cup L\times\{1\}\hookrightarrow L\times J$ is a categorical anodyne extension. Denote by $j$ the restriction of  $\tilde{H}$ to $L\times\{0\}$. By construction, we get a commutative triangle 
    	$$
    	\begin{tikzcd}
    		& \C \\
    		K\arrow[ur, "f"] \arrow[r, hook, "i"] & L.\arrow[u, "j"']
    	\end{tikzcd}
    	$$
    	Since the map $j$ is $J$-homotopic to $p$, it is a Joyal equivalence, and therefore our proof is concluded.
    \end{proof}
	
	\begin{proposition}\label{finssetfininfcat}
		Let $\gamma: \sSet\rightarrow\infcat$ be the localization functor. Then an $\infty$-category $\C$ is finite if and only if there exists a finite simplicial set $K$ and an equivalence $\C\simeq \gamma(K)$. 
	\end{proposition}
	
	\begin{proof}
		Let $\sSet^{f}$ be the full subcategory of $\sSet$ spanned by the finite simplicial sets, and denote by $\cate{F}$ the essential image of the restriction of $\gamma$ to $\sSet^{f}$. We need to show that $\infcat^f$ coincides with $\cate{F}$. 
		
		Let $K$ be any finite simplicial set, so that in particular there exists some finite $n$ such that $K = \text{sk}_n(K)$. By induction on $n$ and using the cellular decomposition 
		$$
		\begin{tikzcd}
			\coprod\limits_{\partial\Delta^n\rightarrow K}\partial\Delta^n \arrow[r] \arrow[d, hook] & \text{sk}_{n-1}(K) \arrow[d] \\
			\coprod\limits_{\Delta^n\rightarrow K}\Delta^n \arrow[r] & \text{sk}_{n}(K) \arrow[ul, phantom, "\usebox\pushout" , very near start, color=black]
		\end{tikzcd}
		$$
		we see that to prove that $L(K)$ belongs to $\infcat^f$, it suffices to show that each $\Delta^n$ does. But this is clear because the $n$-simplex is Joyal equivalent to the $n$-spine. Thus we have $\cate{F}\subseteq\infcat^f$.
		
		Since $\cate{F}$ contains $\emptyset$, $\Delta^0$ and $\Delta^1$, we are now only left to show that $\cate{F}$ is closed under pushouts. Let
		$$\D\leftarrow\C\rightarrow\E$$
		be any cospan of $\infty$-categories in $\cate{F}$, and let $K\rightarrow\C$ be any Joyal equivalence, where $K$ is a finite simplicial set. Thus, by applying \cref{basti} twice, we get a map of cospans
		$$
		\begin{tikzcd}
			L\arrow[d] & K\arrow[l, hook']\arrow[r, hook]\arrow[d] & M\arrow[d] \\
			\D & \C\arrow[l]\arrow[r] & \E
		\end{tikzcd}
		$$
		where the vertical arrows are Joyal equivalences and the upper horizontal arrows are monomorphisms. We then get a Joyal equivalence between the respective homotopy pushouts, and thus the desired conclusion.
	\end{proof}

	\subsection{Finiteness properties of compact conically smooth spaces}
	
	The main goal of this section is to show that the exit paths $\infty$-category of a compact conically smooth stratified space is finite (\cref{finexit}). For this purpose, we make use of Lurie's model of the exit paths $\infty$-category, of which we now recall the definition for the reader's convenience. 
	
	By a slight abuse of notation, for a poset $P$ we still denote by $P$ the topological space obtained by equipping the poset with the Alexandroff topology. If $X\rightarrow P$ is a stratified topological space, then we define $\exit{P}{X}$ by forming the pullback
	
	$$
	\begin{tikzcd}
		\exit{P}{X}\arrow[r] \arrow[d] \arrow[dr, phantom, "\usebox\pullback" , very near start, color=black] & \sing{X} \arrow[d] \\
		N(P) \arrow[r] & \sing{P}
	\end{tikzcd}
	$$
	in the category of simplicial sets. Lurie showed that if the stratification $X\rightarrow P$ is conical, then $\exit{P}{X}$ is an $\infty$-category (\cite[Theorem A.6.4]{lurie2017higher}).
	
	\begin{example}\label{exampleexit}
		Let us give an example of the exit paths $\infty$-category of a stratified topological space for the reader's convenience. Consider the stratified space $\rnum^2\rightarrow\{0<1\}$, where the closed stratum is given by the origin. This can be identified, up to stratified homeomorphism, with the cone on $S^{1}$, with its natural stratification. Using \cref{exitcone}, one can show that $\exit{\{0<1\}}{\rnum^2}$ is equivalent to the $\infty$-category $(B\mathbb{Z})^{\triangleleft}$. This is given by formally adding an initial object to the classifying space $B\mathbb{Z}$. More explicitely, let us denote by $x$ the origin, and $y$ any point different from $x$. Then one can describe $\exit{\{0<1\}}{\rnum^2}$ as a 1-category with two distinct objects $x$ and $y$, where $x$ is initial and the monoid of endomorphisms of $y$ is given by $\mathbb{Z}$.
	\end{example}
	
	\begin{remark}
		In what follows, we often consider a stratified space $X$ without specifying any particular notation for its stratifying poset. In that case, by a slight abuse of notation, we write $\exit{}{X}$ for the $\infty$-category of exit paths of $X$.
	\end{remark}
	
	In \cite[Definition 1.1.5]{ayala2017local}, the authors propose an alternative model of the opposite of the $\infty$-category of exit paths of a conically smooth stratified space, called the \textit{enter paths $\infty$-category}. Let us briefly recall this definition, as it also allow us to introduce some notations that are used throughout our paper.
	
	Let Snglr be the 1-category whose objects are conically smooth stratified spaces and morphisms are conically smooth open immersions, and Bsc the full subcategory of Snglr spanned by \textit{basic} conically smooth stratified spaces, i.e. those which are isomorphic to one of the type $\rnum^n\times C(Z)$, where $Z$ is compact and conically smooth.
	
	In \cite[Lemma 4.1.4]{ayala2017local}, the authors show that Snglr (and therefore Bsc) admits an enrichment in Kan complexes. By passing to homotopy coherent nerves, one gets $\infty$-categories that we denote by 
	\begin{equation}\label{catsofstratspaces}
		\cate{B}\text{sc}\rightarrow \cate{S}\text{nglr}.
	\end{equation}
	
	For any conically smooth stratified space, the authors of \cite{ayala2017local} then define $\text{Entr}(X)$ as the slice $\cate{B}\text{sc}_{/X}$. More precisely, this is defined to be the pullback 
	$$
	\begin{tikzcd}
		\text{Entr}(X)\arrow[r]\arrow[d]\arrow[dr, phantom, "\usebox\pullback" , very near start, color=black] & \cate{S}\text{nglr}_{/X}\arrow[d] \\
		\cate{B}\text{sc}\arrow[r] & \cate{S}\text{nglr}.
	\end{tikzcd}
	$$ 
	Their proof of the exodromy equivalence for constructible sheaves of spaces, combined with the one in \cite{lurie2017higher}, implies that Lurie's exit paths $\infty$-category has to be equivalent to $\text{Entr}(X)\op$ (see \cite[Corollary 1.2.10]{ayala2017local}). 
	
	In this section, we prefer to use Lurie's model, because it has an evident much richer functoriality. Indeed, one sees by the functoriality of $\singularset$ that $\exitpath$ is functorial with respect to general stratified maps. On the other hand, the one in \cite{ayala2017local} is only functorial with respect to conically smooth open embeddings. We also see immediately that, if we stratify $P$ over itself through the identity, then $\exit{P}{P} = N(P)$.  
	
	\begin{definition}
		Let $f :(X\rightarrow P)\rightarrow(Y\rightarrow Q)$ be a map of stratified spaces. We say that $f$ is a \textit{full inclusion of strata} if the underlying map of posets $P\rightarrow Q$ is injective and full, and the square
	$$
	\begin{tikzcd}
		X\arrow[r]\arrow[d] & Y\arrow[d] \\
		P\arrow[r] & Q
	\end{tikzcd}
    $$
	is a pullback of topological spaces.
	\end{definition}

    We will also need the following lemma.

	\begin{lemma}\label{consecutiveff}
		Let $X\rightarrow P$ and $Y\rightarrow Q$ be stratified spaces, and assume that the stratification on $X$ is conical. Assume that we have a stratified embedding $Y\hookrightarrow X$ which is a full inclusion of strata. Then $\exit{Q}{Y}$ is an $\infty$-category and the induced functor $\exit{Q}{Y}\rightarrow \exit{P}{X}$ is fully faithful.
	\end{lemma}
	
	\begin{proof}
		Since the functor $\singularset$ from topological spaces to simplicial sets preserves limits and since $Y\hookrightarrow X$ is a full inclusion of strata, we get a pullback square
		\begin{equation}\label{pbfullinclusion}
			\begin{tikzcd}
			\exit{Q}{Y} \arrow[r] \arrow[d] \arrow[dr, phantom, "\usebox\pullback" , very near start, color=black] & \exit{P}{X} \arrow[d] \\
			N(Q) \arrow[r] & N(P)
		\end{tikzcd}
		\end{equation}
		 of simplicial sets. By \cite[Theorem A.6.4, (1)]{lurie2017higher} the functor $\exit{P}{X}\rightarrow N(P)$ is an inner fibration. Using the pullback square (\ref{pbfullinclusion}), one sees that the map $\exit{Q}{Y}\rightarrow N(Q)$ is also an inner fibration, which implies that $\exit{Y}{Y}$ is an $\infty$-category. Moreover, since the inclusion is full, we know that the functor $N(Q)\rightarrow N(P)$ is fully faithful, and thus we may conclude again by (\ref{pbfullinclusion}).
	\end{proof}
	
	Recall that for a proper conically smooth fiber bundle $\pi: L\rightarrow X$, we define the \textit{fiberwise cone} of $\pi$ as the pushout
	\begin{equation}\label{fiberwisecone}
		\begin{tikzcd}
			L \arrow[d] \arrow[r] & L\times\rnum_{\geq 0} \arrow[d] \\
			X \arrow[r] & C(\pi) \arrow[ul, phantom, "\usebox\pushout" , very near start, color=black]
		\end{tikzcd}
	\end{equation}
	taken in the category of conically smooth stratified spaces (see \cite[Example 3.6.3]{ayala2017local}). By definition, we get a new fiber bundle $C(\pi)\rightarrow X$ whose fibers are isomorphic to basics. We now show how to compute the exit paths of $C(\pi)$ in terms of $L$ and $X$.
	
	\begin{lemma}\label{exitcone}
		Let $\pi: L\rightarrow X$ be a proper conically smooth fiber bundle. Then the commutative square 
		\begin{equation*}
			\begin{tikzcd}
				\exit{}{L} \arrow[d] \arrow[r] & \exit{}{L\times\rnum_{\geq 0}} \arrow[d] \\
				\exit{}{X} \arrow[r] & \exit{}{C(\pi)} 
			\end{tikzcd}
		\end{equation*}
		in $\infcat$ induced by (\ref{fiberwisecone}) is a pushout.
	\end{lemma}
	
	\begin{proof}
		By the Van Kampen theorem for exit paths \cite[Theorem A.7.1]{lurie2017higher}, we may assume that $X$ is a basic. Thus, by \cite[Corollary 7.1.4]{ayala2017local}, we may also assume that $\pi$ is a trivial bundle. Since $\exitpath$ commutes with finite products, we may assume that $X = \ast$, and hence we are only left to prove that for any compact conically smooth space $L$ the square
		\begin{equation*}
			\begin{tikzcd}
				\exit{}{L} \arrow[d] \arrow[r] & \exit{}{L\times\rnum_{\geq 0}} \arrow[d] \\
				\Delta^0 \arrow[r] & \exit{}{C(L)} 
			\end{tikzcd}
		\end{equation*}
		is a pushout in $\infcat$. This follows from \cite[Lemma 6.1.4]{ayala2017local}.
	\end{proof}
	
	We will also need to use the \textit{unzip} and \textit{link} construction, as defined in \cite[Definition 7.3.11]{ayala2017local}. By \cite[Proposition 7.3.10]{ayala2017local}, for any proper constructible embedding $X\hookrightarrow Y$ we have a pullback square
    \begin{equation}\label{blowingup}
		\begin{tikzcd}
			\link{X}{Y} \arrow[d, "{\pi_{X}}"] \arrow[r, hook]  \arrow[dr, phantom, "\usebox\pullback" , very near start, color=black] & \unzip{X}{Y} \arrow[d] \\
			X \arrow[r, hook] & Y.
		\end{tikzcd}
	\end{equation}
	Here $\unzip{X}{Y}$ is a conically smooth manifold with corners whose interior is identified with $Y\setminus X$, and $\unzip{X}{Y}\rightarrow Y$ and $\pi_{X} : \link{X}{Y}\rightarrow X$ are proper constructible bundles. 
	
	\begin{example}
		To get a feeling of how $\unzip{X}{Y}$ works, one may think of it as a generalization of the spherical blow-up (see \cite{arone2010functoriality}). More precisely, when $Y$ is a smooth manifold stratified with a closed submanifold $X$ and its open complement, then the unzip of $X\hookrightarrow Y$ coincides with the spherical blow-up of $X$ in $Y$, and the link is diffeomorphic to the boundary of any normalized tubular neighbourhood of $X$ in $Y$. 
		
		For a visual representation of the unzip construction, we refer to \cite[Figure 7.3.1]{ayala2017local}.
	\end{example}
	
	The link of a proper constructible embedding is used to provide tubular neighbourhoods in the stratified setting. In \cite[Proposition 8.2.5]{ayala2017local}, the authors show that there is a conically smooth map
	\begin{equation}\label{tubneigh}
		C(\pi_X) \hookrightarrow Y
	\end{equation}
	under $X$ which is a refinement onto its image and whose image is open in $Y$. Here we are using the same notations as in (\ref{blowingup})). Denote by $\reallywidetilde{\link{X}{Y}}\times\rnum_{>0}$ and $\reallywidetilde{C(\pi_X)}$ the respective refinements of $\link{X}{Y}\times\rnum_{>0}$ and $C(\pi_X)$ through the embedding (\ref{tubneigh}).
	
	\begin{corollary}\label{tubularcollar}
		Let $X\hookrightarrow Y$ be a proper conically smooth constructible embedding. Then the square
		$$
		\begin{tikzcd}
			\exit{}{\reallywidetilde{\link{X}{Y}}\times\rnum_{>0}} \arrow[r] \arrow[d] & \exit{}{Y\setminus X} \arrow[d] \\
			\exit{}{\reallywidetilde{C(\pi_X)}} \arrow[r] & \exit{}{Y}
		\end{tikzcd}
		$$
		is a pushout in $\infcat$.
	\end{corollary}

    \begin{proof}
    	This follows immediately by the Van Kampen theorem for exit paths $\infty$-categories in \cite[Theorem A.7.1]{lurie2017higher}. 
    \end{proof}
	
	\begin{remark}\label{tubneighstrata}
		Notice that, for the existence of tubular neighbourhoods, one may relax the assumption of properness for a constructible embedding $i:X\hookrightarrow Y$ to just requiring that there is a factorization
		$$
		\begin{tikzcd}
			& Y' \arrow[dr, "j", hook] & \\
			X \arrow[ur, "i'", hook] \arrow[rr, "i"] &  & Y
		\end{tikzcd}
		$$
		where $i'$ is a proper constructible embedding and $j$ is a conically smooth open embedding. For example, if $P$ is the stratifying poset of $Y$ and $X = Y_{\alpha}$ for some $\alpha\in P$, one may pick $Y' = Y_{\geq \alpha}$ and thus get a tubular neighbourhood of $Y_p$.
	\end{remark}
	
	We are now ready to prove the main result of the section.
	
	\begin{proposition}\label{finexit}
		Let $X$ be any compact conically smooth stratified space. Then $\exit{}{X}$ is a finite $\infty$-category.
	\end{proposition}
	
	\begin{proof}
		Since $X$ is compact, $X$ is finite dimensional, and hence also has finite depth. We then argue by induction on $\dpt{X} = k$.
		
		If $k = 0$, it is well known that $\exit{}{X} = \sing{X}$ is a finite $\infty$-groupoid. For example, this follows by Van Kampen theorem \cite[Theorem A.3.1]{lurie2017higher} and the existence of finite good covers for $X$.
		
		Assume now that $k$ is positive. Denote by $X_0$ the union of strata of minimal depth, and by $X_{>0}$ its complement in $X$. One sees that $X_{>0}\hookrightarrow X$ is a proper constructible embedding, and hence by \cref{tubularcollar} we get a pushout 
		$$
		\begin{tikzcd}
			\exit{}{\reallywidetilde{\link{>0}{X}}\times\rnum_{>0}} \arrow[r] \arrow[d] & \exit{}{X_0} \arrow[d] \\
			\exit{}{\reallywidetilde{C(\pi_{>0})}} \arrow[r] & \exit{}{X} \arrow[ul, phantom, "\usebox\pushout" , very near start, color=black].
		\end{tikzcd}
		$$  
		By \cref{exitcone}, to conclude the proof it suffices to show that $\exit{}{\reallywidetilde{\link{>0}{X}}}$, $\exit{}{\reallywidetilde{C(\pi_{>0})}}$, and $\exit{}{X_0}$ are finite. 
		
		Being a closed subset of $X$, the space $X_{>0}$ is compact. By the pullback square (\ref{blowingup}) $\link{>0}{X}$ is compact too. Since the depths of both are stricly less than $\dpt{X}$, by the inductive hypothesis we get that $X_{>0}$ and $\link{>0}{X}$ both have finite exit paths $\infty$-categories. Notice that we have a stratified embedding $$\reallywidetilde{\link{>0}{X}}\times\rnum_{>0}\hookrightarrow X_0$$ and $X_0$ is a smooth manifold. In particular, we see that the stratification on $\reallywidetilde{\link{>0}{X}}$ is trivial. Thus $\exit{}{\reallywidetilde{\link{>0}{X}}}$ is a localization of $\exit{}{\link{>0}{X}}$ at all maps, and by \cref{finloc} we see that $\exit{}{\reallywidetilde{\link{>0}{X}}}$ is finite. 
		
		By \cref{exitcone} and the inductive hypothesis, we also know that $\exit{}{C(\pi_{>0})}$ is finite. Using \cite[Proposition 1.2.13]{ayala2017local}, we see that the canonical functor $$\phi : \exit{}{C(\pi_{>0})}\rightarrow\exit{}{\reallywidetilde{C(\pi_{>0})}}$$ is a localization at the class of exit paths that are inverted by $\phi$. Since the inclusion $C(\pi_{>0})\rightarrow X$ lies under $X>0$, the same argument as above shows that a non-invertible exit path is inverted by $\phi$ if and only if it lies inside $\link{>0}{X}\times \rnum_{>0}\hookrightarrow C(\pi_{>0})$. Notice that $\link{>0}{X}\times \rnum_{>0}\hookrightarrow C(\pi_{>0})$ is a consecutive inclusion of strata. Therefore, by \cref{consecutiveff} the induced functor on exit paths is the inclusion of a full subcategory. This implies that one can identify $\exit{}{\reallywidetilde{C(\pi_{>0})}}$ with the localization of $\exit{}{C(\pi_{>0})}$ at $\exit{}{\link{>o}{X}}\times\rnum_{>0}$. Hence, by \cref{finloc}, $\exit{}{\reallywidetilde{C(\pi_{>0})}}$ is finite as well.
		
		We know that $X_0$ is the interior of the compact manifold with corners $\unzip{>0}{X}$. One can show that the existence of collaring for corners (\cite[Lemma 8.2.1]{ayala2017local}) implies that the inclusion $X_0\hookrightarrow\unzip{>0}{X}$ is a homotopy equivalence. In the proof of \cite[Lemma 2.1.3]{ayala2019stratified} one may find a construction of a homotopy inverse of the inclusion $X_0\hookrightarrow\unzip{>0}{X}$ in the special case where the corner structure is a boundary. However, the construction of such an inverse in the more general case is completely analogous. Hence to conclude the proof it suffices to show that $\sing{\unzip{>0}{X}}$ is finite. This follows by the existence of good covers for manifolds with corners. 
	\end{proof}
	
	\begin{corollary}
		Let $X$ be a finitary conically smooth stratified space (see \cite[Definition 8.3.6]{ayala2017local}). Then $\exit{}{X}$ is a finite $\infty$-category. In particular, if $X$ is the interior of a compact conically smooth manifold with corners, then $\exit{}{X}$ is a finite $\infty$-category.
	\end{corollary}
	
	\begin{proof}
		By \cref{finexit} and \cref{exitcone}, the class of conically smooth spaces with finite exit paths $\infty$-category contains all basics. Thus it suffices to show that it is closed under taking collar glueings. Suppose there is a collar glueing $f\colon Y\rightarrow[-1, 1]$ such that $f^{-1}([-1,1))$  $f^{-1}((-1, 1])$ and $f^{-1}(0)$ are finitary. Then we get an open covering of $Y$ given by $f^{-1}([-1,1))$  $f^{-1}((-1, 1])$ and $\rnum \times f^{-1}(0)$, which implies that $\exit{}{Y}$ is finite by Van-Kampen.
		
		The last part of the statement follows by \cite[Theorem 8.3.10 (1)]{ayala2017local}.
	\end{proof}

	\section{Homotopy invariance and exodromy with general coefficients}
	
	In this section we explain how to use \cite{volpe2021six} to prove \textit{homotopy invariance} and the \textit{exodromy equivalence} for (formally) contructible sheaves (see \cref{constrdef}) valued in stable and bicomplete $\infty$-categories (\cref{exodromy}). 
	
	A proof of homotopy invariance for constructible sheaves with presentable coefficients can be found in \cite{Haine2020TheHO}. Our argument in \cref{constrdef} follows precisely the one in \cite{Haine2020TheHO}. Nevertheless, we will try to quickly outline the main steps to convince the reader that all the results in \cite{Haine2020TheHO}, after \cite{volpe2021six}, generalize to the setting of stable bicomplete coefficients, at least if we restrict ourselves to locally compact Hausdorff spaces.
	
	The exodromy equivalence was first proven by Lurie in \cite[Theorem A.9.3]{lurie2017higher}, and later generalized in \cite{porta2022topological}. For constructible sheaves of $\infty$-groupoids on conically smooth stratified spaces, this was proven in \cite[Theorem 1.2.5]{ayala2017local}. Here we stick with conically smooth stratified spaces, and we provide a short argument that works for constructible sheaves in stable and bicomplete $\infty$-categories. This is essentially a combination of the homotopy invariance and \cite[Lemma 4.5.1]{ayala2017local}.
	
	With these at hand, we show that global sections of constructible sheaves on compact conically smooth stratified spaces are dualizable (\cref{globsecconstrshcpt}).
	
	\subsection{Homotopy invariance of constructible sheaves}
	
	From now on, all $\infty$-categories appearing as coefficients for sheaves will be assumed to be stable and bicomplete, all topological spaces locally compact Hausdorff and all posets Noetherian. The next lemma shows in particular that the stratifying poset of any $C^0$-stratified space (see \cite[Definition 2.1.15]{ayala2017local}) is Noetherian. Since any conically smooth stratified space is by definition $C^0$-stratified, we see that all our stratified spaces of interest will have Noeterian stratifying posets.
	
	\begin{lemma}
		Let $X\rightarrow P$ be a $C^0$-stratified space. Then $P$ is locally finite, and therefore Noetherian.
	\end{lemma}

    \begin{proof}
    	Recall that, by \cite[Lemma 2.2.2]{ayala2017local}, any $C^0$-stratified space admits a basis given by its open subsets isomorphic as stratified spaces to one of the type $\mathbb{R}^n\times C(Z)$, where $Z$ is a compact $C^0$-stratified space. Therefore, it will suffice $P$ is finite when $X$ is compact. We prove this by induction on the depth of $X$.
    	
    	When the depth of $X$ is $0$, it follows by \cite[Lemma 2.22]{nocera2021whitney} that $P$ is discrete. Since $X$ compact, it can have only a finite number of connected components, and therefore $P$ has to be finite.
    	
    	Assume that the depth of $X$ is $n$. Since $X$ is compact and $C^0$-stratified, one may find a finite cover of $X$ by open subsets isomorphic as stratified spaces to $\mathbb{R}^n\times C(Z)$, where $Z$ is a compact $C^0$-stratified. But the depth of $Z$ is smaller than $n$, and thus by inductive assumption its stratifying poset has to be finite. Therefore, $P$ has to be finite. 
    \end{proof}
	
	\begin{definition}\label{constrdef}
		Let $X$ be any locally compact Hausdorff topological space, and let $a:X\rightarrow\ast$ be the unique map. We say that a sheaf $F\in\Sh{X}{\C}$ is \textit{constant} if there exists an object $M\in\C$ and an equivalence $F\simeq\pb{a}_{\C}M$. We say that $F$ is \textit{locally constant} if there exists an open covering $\{U_i\}_{i\in I}$ of $X$ such that $F_{|_{U_i}}$ is constant.
		
		Let $X\rightarrow P$ be a stratified space. We say that a sheaf $F\in\Sh{X}{\C}$ is \textit{formally constructible} if for any $\alpha\in P$ the restriction of $F$ to the stratum $X_{\alpha}$ is locally constant. 
		
		Assume now that $\C$ admits a closed symmetric monoidal structure, and denote by $\C^{\text{dual}}$ the full subcategory of $\C$ spanned by dualizable objects. We say that $F$ is \textit{constructible} if $F$ is formally constructible and each stalk of $F$ belongs to $\C^{\text{dual}}$. 
		
		We denote by $\fcSh{X}{C}$ and $\cSh{X}{\C}$ the full subcategories of $\Sh{X}{\C}$ spanned respectively by formally constructible and constructible sheaves. Dually, we define formally constructible and constructible cosheaves on $X$ as $\fccoSh{X}{\cate{C}}\coloneqq\fcSh{X}{\cate{C}\op}\op$ and $\ccoSh{X}{\cate{C}}\coloneqq\cSh{X}{\cate{C}\op}\op$.
	\end{definition}

	In this paper we only deal with constructible sheaves with respect to a specified stratification. Therefore, we will take the liberty of omitting the stratifying poset from our notation for constructible sheaves.
	
	We first recall the proof of the homotopy invariance of constructible sheaves. For the reader's convenience, we will also recall the definition of stratified homotopy equivalence.
	
	\begin{definition}
		Let $X\rightarrow P$ and $Y\rightarrow Q$ be stratified spaces, and let $[0,1]\subseteq\rnum$ be the closed interval, considered as a stratified space with a single stratum. A \textit{stratified homotopy} is a map of stratified spaces $H: X\times[0,1]\rightarrow Y$. 
		
		We say that a stratified map $f : X \rightarrow Y$ is a \textit{stratified homotopy equivalence} if there exists a stratified map $g: Y\rightarrow X$ and stratified homotopies $H:X\times[0,1]\rightarrow X$, $K:Y\times[0,1]\rightarrow Y$ such that $H_{|{X\times\{0\}}} = \text{id}_X$, $H_{|{X\times\{1\}}} = gf$, $K_{|{Y\times\{0\}}} = \text{id}_Y$ and $K_{|{Y\times\{1\}}} = fg$.
	\end{definition}
	
	\begin{theorem}[Homotopy invariance]\label{htpyinv}
		Let $X\rightarrow P$ be a stratified space. Let $p : X\times [0,1]\rightarrow X$ be the canonical projection. Then $\pb{p} : \Sh{X}{\C}\rightarrow\Sh{X\times [0,1]}{\C}$ restricts to an equivalence
		$$\fcSh{X}{\C}\simeq\fcSh{X\times [0,1]}{\C}.$$
		As a consequence, if $Y\rightarrow P$ is another stratified space and $f: X\rightarrow Y$ is a stratified homotopy equivalence, then the functor
		$$\pb{f}:\fcSh{Y}{\C}\rightarrow\fcSh{X}{\C}.$$
		is an equivalence.
	\end{theorem}

    \begin{proof}
    	We first treat the case of locally constant sheaves, i.e. when $P = \ast$. By \cite[Lemma A.2.9]{lurie2017higher} and \cite[Corollary 5.2]{volpe2021six} we know that $\pb{p}$ is fully faithful. We start by showing that $\pf{p}$ preserves constant sheaves. If $a : X\rightarrow\ast$ and $b : X\times [0,1] \rightarrow \ast$ are the unique maps then, for any object $M\in\C$, the fully faithfulness of $\pb{p}$ implies that we have equivalences
    	\begin{align*}
    		\pf{p}\pb{b}M&\simeq \pf{p}\pb{p}\pb{a}M \\
    		& \simeq \pb{a}M.
    	\end{align*}
        
        Let now $F$ be any locally constant sheaf on $X\times[0,1]$. By \cite[Lemma 4.9]{Haine2020TheHO}, there exists an open cover $\{U_i\}_{i\in I}$ of $X$ such that $F_{|_{U_i\times[0,1]}}$ is constant. Therefore, since $[0, 1]$ is compact, by proper base change (see\cite[Proposition 6.1]{volpe2021six}) we see that $(\pf{p}F)_{|_{U_i}}$ is constant. Hence we see that $\pf{p}$ preserves locally constant sheaves. Thus to conclude we only need to show that, for any locally constant sheaf $F$ on $X\times [0, 1]$, the counit map $\pb{p}\pf{p}F\rightarrow F$ is an equivalence. 
        
        Again by basechange and \cite[Lemma 4.9]{Haine2020TheHO}, we may reduce to the case when $F\simeq\pb{b}M$ is constant. In this case we have a commutative diagram
        $$
        \begin{tikzcd}[column sep= huge]
        	\pb{p}\pf{p}\pb{b}M \arrow[r, "{\text{counit}_{\pb{b}M}}"] \arrow[d, "\simeq"] & \pb{b}M \arrow[d, "\simeq"] \\
        	\pb{p}\pf{p}\pb{p}\pb{a}M \arrow[r, "{\text{counit}_{\pb{p}\pb{a}M}}"] & \pb{p}\pb{a}M \\
        	\pb{p}\pb{a}M \arrow[u, "{\pb{p}(\text{unit}_{\pb{a}M})}", "\simeq"'] \arrow[ur, "\text{id}"'] & 
        \end{tikzcd}
        $$
        which implies the desired result.
    	
    	Now assume that $P$ is any Noetherian poset. Using base change in a similar way as before, one sees that $\pf{p}$ preserves formally constructible sheaves, and thus we are left to show that for any $F$ constructible $\pb{p}\pf{p}F\rightarrow F$ is an equivalence. By \cite[Corollary 4.2]{volpe2021six} any stable and bicomplete $\infty$-category respects glueing in the sense of \cite[Definition 5.17]{Haine2020TheHO}. Hence \cite[Lemma 5.19]{Haine2020TheHO} implies that the functor given by restricting to the strata of $X$ are jointly conservative. By base change we may thus assume that $F$ is locally constant, in which case the counit is known to be an equivalence by the previous step.
    	
    	The last part of the statement then follows by a standard argument analogous to the proof of \cite[Corollary 3.1]{volpe2021six}.
    \end{proof}

    We now present a couple of useful corollaries of homotopy invariance.
    
    \begin{corollary}\label{htpyinvglobsec}
    	Let $f:X\rightarrow Y$ be a stratified homotopy equivalence, and $F\in\fcSh{Y}{\C}$. The natural map 
    	$$ \Sec{Y}{F}\rightarrow\Sec{X}{\pb{f}F}$$
    	is an equivalence.
    \end{corollary}
    
    \begin{proof}
    	The commutative triangle 
    	$$
    	\begin{tikzcd}
    		X \arrow[rr, "f"] \arrow[rd, "a"'] & & Y \arrow[ld, "b"] \\
    		& \ast &
    	\end{tikzcd}
    	$$
    	induces an invertible natural tranformation $\pb{f}\pb{b}\simeq\pb{a}$. Since both $\pb{a}$ and $\pb{b}$ factor through formally constructible sheaves, we get $\pb{b}\simeq\eta\pb{a}$, where $\eta$ is any adjoint inverse of the restriction of $\pb{f}$ to $\fcSh{Y}{\C}$. Thus, by passing to right adjoints, we get the desired equivalence.
    \end{proof}
    We will need the following lemma.
    
    \begin{lemma}\label{basisconept}
    	Let $Z$ be a compact topological space. Let $\mathbb{R}_{\geq 0}\times Z\rightarrow C(Z)$ be the quotient map, and for each $\epsilon>0$ denote by $C_{\epsilon}(Z)$ the image of the open subset $[0,\epsilon)\times Z$. Then the family of open subsets
    	$$\{C_{\epsilon}(Z)\mid \epsilon\in\mathbb{R}_{>0}\}$$ forms a basis at the cone point.
    \end{lemma}

    \begin{proof}
    	We will prove that, for every open subset $W$ of $\mathbb{R}_{\geq 0}\times Z$ containing $\{0\}\times Z$, there exists some $\epsilon > 0$ such that $[0, \epsilon)\times Z\subseteq W$. Since $Z$ is compact, one can obtain a finite covering of $\{0\}\times Z$ with opens of the type $[0,\epsilon_i)\times V_i\subseteq W$, and thus by taking $\epsilon$ to be the minimum of the $\epsilon_i$ we get the claim.
    \end{proof}

    \begin{corollary}\label{stalksconstr}
    	Let $X$ be a $C^0$-stratified space and $F\in\fcSh{X}{\C}$. For any point $x\in X$ and any conical chart $\rnum^n\times C(Z)$ centered at $x$, the natural map
    	$$\Sec{\rnum^n\times C(Z)}{F}\rightarrow F_x$$
    	is an equivalence.
    \end{corollary}

    \begin{proof}    	
    	First of all, notice that there is a homeomorphism $\rnum^n\cong C(S^{n-1})$ sending $0$ to the cone point, and under which the subsets $C_{\epsilon}(S^{n-1})$ on the right-hand side are identified with open balls centered at zero and with radius $\epsilon$. Let us denote these subsets as $B_{\epsilon}(0)$. By \cref{basisconept}, the family of open subsets of the type $\{B_{\epsilon}(0)\times C_{\epsilon}(Z)\}_{\epsilon>0}$ is cofinal in the family of all open subsets of $\rnum^n\times C(Z)$ containing $(0, \text{cone point})$. Hence we see that
    	$$\varinjlim\limits_{\epsilon> 0}\Sec{B_{\epsilon}(0)\times C_{\epsilon}(Z)}{F}\simeq F_x.$$
    	By \cref{htpyinvglobsec} we have
    	\begin{equation*}
    		\Sec{\rnum^n\times C(Z)}{F}\simeq\varinjlim\limits_{\epsilon> 0}\Sec{\rnum^n\times C_{\epsilon}(Z)}{F}
    	\end{equation*}
        which concludes our proof.
    \end{proof}

    Now let $X\rightarrow P$ be a conically smooth stratified space, and let $\alpha\in P$. By \cref{tubneighstrata}, we get commutative triangle
    $$
    \begin{tikzcd}
    	& X_{\alpha} \arrow[ld, "i"', hook'] \arrow[rd, "{i_{\alpha}}", hook] & \\
    	C(\pi_{\alpha}) \arrow[rr, "j"', hook] & & X
    \end{tikzcd}
    $$
    where $\pi_{\alpha}$ is the fiber bundle $\link{X_{\alpha}}{X_{\geq\alpha}}\rightarrow X_{\alpha}$, $i$ is the cone-point section of the fiber bundle $p : C(\pi_{\alpha})\rightarrow X_{\alpha}$ and $j$ is a conically smooth open immersion. For any $F\in\Sh{X}{\C}$, the unit of the adjunction $\pb{i}\dashv\pf{i}$ gives a natural map
    \begin{equation}\label{mapreststrata}
    	\pf{p}\pb{j}F\rightarrow \pf{p}\pf{i}\pb{i}\pb{j}F\simeq\pb{i_{\alpha}}F.
    \end{equation}
    Furthermore the map (\ref{mapreststrata}) can be obtained by applying the sheafification functor to
    \begin{equation}\label{premapreststrata}
    	\pf{p}\pb{j}F\rightarrow \pf{p}\pf{i}(\pb{i})^{pre}\pb{j}F\simeq(\pb{i_{\alpha}})^{pre}F
    \end{equation}
    where $(\pb{i})^{pre}$ and $(\pb{i_{\alpha}})^{pre}$ denote the corresponding presheaf pullback functors.

    \begin{corollary}\label{restrata}
    	Let $X_{\alpha}\hookrightarrow X$ be the inclusion of a stratum in a conically smooth stratified space, and let $F$ be any formally constructible sheaf on $X$. Then the map (\ref{mapreststrata}) is an equivalence.
    \end{corollary}

    \begin{proof}
    	Since $F$ is a sheaf, it suffices to show that (\ref{premapreststrata}) is an equivalence. As usual, it suffices to prove that it is an equivalence after taking sections on any euclidean chart $U$ of $X_{\alpha}$. For any such $U$, by \cite[Corollary 7.1.4]{ayala2017local} we have $$\Sec{U}{\pf{p}\pb{j}F} = \Sec{U\times C(Z)}{F}$$ for some compact conically smooth stratified space $Z$. Thus we are left to show that the natural map
    	\begin{equation}\label{presh}
    		\Sec{U\times C(Z)}{F}\rightarrow\varinjlim\limits_{U\subseteq V}\Sec{V}{F}
    	\end{equation}
    	is an equivalence.
    	
    	By a cofinality argument, the map (\ref{presh}) factors through an equivalence
    	$$\varinjlim\limits_{\epsilon> 0}\Sec{U\times C_{\epsilon}(Z)}{F}\simeq\varinjlim\limits_{U\subseteq V}\Sec{V}{F},$$
    	and thus we are left to show that 
    	$$\Sec{U\times C(Z)}{F}\rightarrow\varinjlim\limits_{\epsilon> 0}\Sec{U\times C_{\epsilon}(Z)}{F}$$
    	is an equivalence. This last assertion then follows by \cref{htpyinvglobsec}.
        \end{proof}

    \subsection{Exodromy}
    
    This subsection is devoted to giving a proof of the exodromy equivalence on conically smooth stratified spaces for constructible sheaves valued in stable and bicomplete $\infty$-categories. To do this we use the model of the exit paths $\infty$-category of a conically smooth stratified space given in \cite[Definition 1.1.5]{ayala2017local}. As an intermediate step, we provide a useful characterization of the property of being formally constructible for a sheaf on a conically smooth stratified space. In short, this says that a sheaf is formally constructible if and only if it is homotopy invariant (see \cref{charconstr} for a precise statement). We also collect a couple of useful corollaries of this fact, that are used later in the following section to prove a crucial step of our main result.
    
    One has functors
    $$
    \begin{tikzcd}
    	\text{Bsc}_{/X}\arrow[r, "\gamma"] \arrow[d, "\text{im}"] & \cate{B}\text{sc}_{/X} \\
    	\cate{U}(X) &
    \end{tikzcd}
    $$
    where $\cate{U}(X)$ denotes the poset of open subsets of $X$, and $\text{im}$ sends an open immersion to its image in $X$. In \cite[Lemma 4.5.1]{ayala2017local}, the authors show that $\gamma$ is a localization at the class $\cate{W}$ of open immersions of basics $U\hookrightarrow V$ such that $U$ and $V$ are abstractly isomorphic in Strat. That is, precomposing with $\gamma$ gives an equivalence
    \begin{equation}\label{locbasics}
    	\pb{\gamma} : \Fun(\exit{}{X},\C)\rightarrow \Fun_{\cate{W}}(\text{Bsc}_{/X}\op, \C)
    \end{equation} 
    where the right-hand side denotes the full subcategory of $\Fun(\text{Bsc}_{/X}\op, \C)$ spanned by functors which send all morphisms in $\cate{W}$ to equivalences. In the next proposition we show that $\cate{W}$ coincides with the class of open immersions which are stratified homotopy equivalences, and then characterize the property of being formally constructible through these maps.

    \begin{proposition}\label{charconstr}
    	Let $X\rightarrow P$ be a conically smooth stratified space, and let $F\in\Sh{X}{\C}$. Then the following assertions are equivalent: 
    	\begin{enumerate}[(i)]
    		\item $F$ is formally constructible;
    		\item for any inclusion $V\hookrightarrow U$ of basic open subsets of $X$ which is a stratified homotopy equivalence, then 
    		$$\Sec{U}{F}\rightarrow\Sec{V}{F}$$
    		is an equivalence;
    		\item for any inclusion $V\hookrightarrow U$ of basic open subsets of $X$ which are abstractly isomorphic, then 
    		$$\Sec{U}{F}\rightarrow\Sec{V}{F}$$
    		is an equivalence.
    	\end{enumerate}
    \end{proposition}

    \begin{proof}
    	We first prove that (ii) is equivalent to (iii) by showing that an open immersion $j : V\hookrightarrow U$ of basic open subsets of $X$ is a stratified homotopy equivalence if and only if $U$ and $V$ are abstractly isomorphic. 
    	
    	First of all, observe that $j$ is conically smooth, because the conically smooth structures of $U$ and $V$ are restricted from the one of $X$. If $j$ is a stratified homotopy equivalence, it follows that $U$ and $V$ are stratified over the same subposet of $P$, and so in particular $V$ intersects the stratum of maximal depth of $U$. Therefore, by the equivalence of the conditions (2) and (4) in \cite[Lemma 4.3.7]{ayala2017local}, $U$ and $V$ are isomorphic. 
    	
    	Viceversa, assume that $U$ is abstractly isomorphic to $V$. Therefore, up to composing with isomorphisms on both sides, we may assume $j$ is of the form $j : \rnum^n\times C(Z)\hookrightarrow\rnum^n\times C(Z)$. Then by \cite[Lemma 4.3.6]{ayala2017local} $j$ is homotopy equivalent to $D_0j$, where $D_0j$ denotes the differential of $j$ at the point $(0, \text{cone pt})$ (see \cite[Definition 3.1.4]{ayala2017local}). Since $D_0j$ is a stratified homotopy equivalence, then the same is true for $j$.
    	
    	By \cref{htpyinv} we have that (i) implies (ii), so we are left to show that (iii) implies (i). Let $i : Y\rightarrow X$ be the inclusion of a stratum, and let $V\hookrightarrow U$ be an inclusion of euclidean charts of $Y$. By \cref{restrata}, the horizontal arrows in the commutative square
    	$$
    	\begin{tikzcd}
    		\Sec{U\times C(Z)}{F} \ar[r] \ar[d] & \Sec{U}{\pb{i}F} \ar[d] \\
    		\Sec{V\times C(Z)}{F} \ar[r] & \Sec{V}{\pb{i}F}
    	\end{tikzcd}
        $$
        are invertible, and thus $\Sec{U}{\pb{i}F}\rightarrow\Sec{V}{\pb{i}F}$ is invertible too. Therefore, to deduce that $\pb{i}F$ is locally constant, we just need to show that (iii) implies (i) in the special case when $X$ has a trivial stratification, i.e. when $X$ is a smooth manifold. The result is now a very special case of \cite[Proposition 3.1]{Haine2020TheHO}. For the reader's convenience, we review and adapt the proof of \cite[Proposition 3.1]{Haine2020TheHO} to our setting in the following proposition.
    \end{proof}

    \begin{proposition}
    	Let $X$ be a smooth manifold, and let $F\in\Sh{X}{\C}$. Then the following assertions are equivalent: 
    	\begin{enumerate}[(i)]
    		\item $F$ is locally constant;
    		\item for any inclusion $V\hookrightarrow U$ of euclidean charts of $X$, the restriction 
    		$$\Sec{U}{F}\rightarrow\Sec{V}{F}$$
    		is an equivalence.
    	\end{enumerate}
    \end{proposition}

    \begin{proof}
    	Since the question is local, we may assume that $X = \rnum^n$, in which case we will prove that condition (ii) implies that $F$ is constant. More precisely, we will show that, if $a : \rnum^n\rightarrow \ast$ is the unique map, the counit morphism
    	\begin{equation}\label{consteuc}
    		\pb{a}\pf{a}F\rightarrow F
    	\end{equation}
    	is an equivalence. Since $\rnum^n$ is hypercomplete and admits a basis given by those open subsets diffeomorphic to itself, it then suffices to check that for any such open $j :U\hookrightarrow \rnum^n$, the map $\pf{a}\pf{j}\pb{j}\pb{a}\pf{a}F\rightarrow\pf{a}\pf{j}\pb{j}F$ obtained by applying to (\ref{consteuc}) the functor of sections at $U$ is invertible. Notice that we have a commutative triangle
    	$$
    	\begin{tikzcd}[column sep= huge]
    		\pf{a}\pf{j}\pb{j}\pb{a}\pf{a}F \arrow[r, "{\pf{a}\pf{j}\pb{j}\text{counit}_F}"]  \arrow[d, "{\text{counit}_{\pf{a}F}}"']& \pf{a}\pf{j}\pb{j}F \\
    		\pf{a}F \arrow[ur, "{\pf{a}(\text{unit}_F)}"', "\simeq"]
    	\end{tikzcd}
    	$$
    	where the diagonal map is invertible by the assumption in (iii). Thus to conclude the proof it suffices to show that $\pb{a}$ is fully faithful, which follows by the homotopy invariance of the shape (see \cite[Corollary 3.1]{volpe2021six}).
    \end{proof}
    
    \begin{corollary}\label{fiberbundlespreserveconstructibility}
    	Let $f:L\rightarrow X$ be a conically smooth fiber bundle. Then the pushforward $\pf{f}^{\C}:\Sh{L}{\C}\rightarrow\Sh{X}{\C}$ preserves formally constructible sheaves.
    \end{corollary}

    \begin{proof}
    	By definition of a conically smooth fiber bundle, for any point $x\in X$, there exist an open neighbourhood $U$ of $x$, a conically smooth stratified space $Y$ and a pullback square
    	$$
    	\begin{tikzcd}
    		U\times Y \arrow[r, hook]\arrow[d, "p"]\arrow[dr, phantom, "\usebox\pullback" , very near start, color=black] & L\arrow[d, "f"] \\
    		U\arrow[r, hook] & X,
    	\end{tikzcd}
    	$$
    	where $p$ is the canonical projection.
    	Let $F$ be a formally constructible sheaf on $L$. To prove that $\pf{f}^{\C}F$ is formally constructible, it suffices to show that its restriction to any $U$ as above is formally constructible. Therefore, since any open immersion gives a locally contractible geometric morphism, by smooth base-change (see \cite[Lemma 3.25]{volpe2021six}) it suffices to show that $\pf{p}^{\C}$ preserves formally constructible sheaves.
    	
    	Let $G$ be any formally constructible sheaf on $U\times Y$, and let $j:V\hookrightarrow W$ be any open immersion of basics in $U$ which is a stratified homotopy equivalence. By \cref{charconstr}, we need to show that the restriction of $\pf{p}^{\C}G$ corresponding to $j$ is an equivalence. We have a commutative square
    	$$
    	\begin{tikzcd}
    		\Sec{W}{\pf{p}^{\C}G}\arrow[r]\arrow[d, "\simeq"] & \Sec{V}{\pf{p}^{\C}G} \arrow[d, "\simeq"] \\
    		\Sec{W\times Y}{G}\arrow[r] & \Sec{V\times Y}{G}.
    	\end{tikzcd}
    	$$
    	Since $G$ is constructible and $j\times\text{id}_Y:V\times Y\hookrightarrow W\times Y$ is again a stratified homotopy equivalence, by \cref{htpyinvglobsec} we see that the lower horizontal arrow is an equivalence. The proof is then concluded by observing that both vertical arrows are equivalences.
    \end{proof}

    \begin{corollary}\label{shriekrestrofformconstrislocconst}
    	Let $X$ be a conically smooth stratified space, and let $i:X_{\alpha}\hookrightarrow X$ be the inclusion of a stratum. Let $F\in\Sh{X}{\C}$ be a formally constructible sheaf. Then $\pbp{i}_{\C}F$ is locally constant. 
    \end{corollary}

    \begin{proof}
    	We have a fiber sequence
    	\begin{equation}\label{fibclosedstratum}
    		\pbp{i}_{\C}F\rightarrow\pb{i}_{\C}F\rightarrow\pb{i}_{\C}\pf{j}^{\C}\pb{j}_{\C}F
    	\end{equation}
    	where $j$ is the open immersion $X_{>\alpha}\hookrightarrow X$.
    	Hence to conclude it suffices to show that $\pb{i}_{\C}\pf{j}^{\C}\pb{j}_{\C}F$ is locally constant.
    	
    	Let $\pi_{\alpha}:\link{X_{\alpha}}{X_{\geq\alpha}}\rightarrow X_{\alpha}$ be the projection from the link of $X_{\alpha}$ in $X_{\geq\alpha}$. Denote by $k$ the open immersion 
    	$$k:\link{X_{\alpha}}{X_{\geq\alpha}}\times\rnum_{>0}\hookrightarrow C(\pi_{\alpha})\hookrightarrow X,$$
    	and by $p$ the conically smooth fiber bundle
    	$$p:\link{X_{\alpha}}{X_{\geq\alpha}}\times\rnum_{>0}\rightarrow\link{X_{\alpha}}{X_{\geq\alpha}}\xrightarrow{\pi_{\alpha}}X_{\alpha},$$
    	where the first arrow is the canonical projection. Then, by \cref{restrata}, we have an equivalence $\pb{i}_{\C}\pf{j}^{\C}\pb{j}_{\C}F\simeq\pf{p}^{\C}\pb{k}_{\C}F.$ But since $\pb{k}_{\C}F$ is formally constructible, and $p$ is a conically smooth fiber bundle, we can apply \cref{fiberbundlespreserveconstructibility} to deduce that $\pb{i}_{\C}\pf{j}^{\C}\pb{j}_{\C}F$ is locally constant, and therefore the proof is concluded.	
    \end{proof}

    \begin{theorem}[Exodromy]\label{exodromy}
    	The composition 
    	$$
    	\begin{tikzcd}
    		\Fun(\exit{}{X}, \C) \arrow[r, "\pb{\gamma}"] & \Fun_{\cate{W}}((\text{Bsc}_{/X})\op, \C) \arrow[r, "\pf{\text{im}}"] & \Fun(\cate{U}(X)\op, \C)
    	\end{tikzcd}
    	$$
    	is fully faithful with essential image $\fcSh{X}{\C}$. Moreover, if we assume that $\C$ has a closed symmetric monoidal structure, the statement remains true if we replace $\Fun(\exit{}{X}, \C)$ by $\Fun(\exit{}{X}, \C^{\text{dual}})$ and $\fcSh{X}{\C}$ by $\cSh{X}{\C}$.
    \end{theorem}

    \begin{proof}
    	By \cref{charconstr}, it suffices to show that the restriction of $\pf{\text{im}}$ to $\Fun_{\cate{W}}(\text{Bsc}_{/X}, \C)$ factors through $\fcSh{X}{\C}$. 
    	
    	Let $U$ be a basic open subset of $X$ and let $\kappa : T\hookrightarrow\text{Bsc}_{/U}$ be a covering sieve. There is at least one $V\in T$ whose image in $U$ intersects the deepest stratum. By the equivalence of conditions (2) and (4) in \cite[Lemma 4.3.7]{ayala2017local}, we see that $U$ and $V$ are abstractly isomorphic. Then for any $F\in\Fun_{\cate{W}}(\text{Bsc}_{/X},  \C)$ we have a commutative triangle
    	$$
    	\begin{tikzcd}
    		\Sec{U}{F} \arrow[r] \arrow[dr, "\simeq"'] & \varprojlim\limits_{O\in T}\Sec{O}{F} \arrow[d] \\
    		& \Sec{V}{F}
    	\end{tikzcd}
    	$$
    	where the diagonal map is invertible by assumption. By \cite[Proposition 3.2.23]{ayala2017local} open subsets isomorphic to basics form a basis for the topology of $X$. Since $X$ is hypercomplete, by \cite[Theorem A.6]{aoki2020tensor} to show that $\pf{\text{im}}F$ is a sheaf it suffices to check the sheaf condition on open covers formed by basics of basic open subsets. More succinctly, we have to show that the horizontal map in the triangle above is invertible. Moreover, by the 2-out-of-3 property, it suffices to show that the vertical map is invertible. 
    	
    	Let $\delta: T\rightarrow \cate{T}$ be the localization of $T$ at $\cate{W}$. Since $\delta$ is final and $\pb{\kappa}$ sends maps in $X$ to equivalences, the result then follows by observing that $V$ is a terminal object in $\cate{T}$.
    	
    	For the second part of the statement, we just need to show that for any functor $F:\exit{}{X}\rightarrow\C^{\text{dual}}$, the stalks of $\pb{\gamma}\pb{\text{im}}F$ are dualizable. By \cref{stalksconstr}, it suffices to prove that for any open $U$ which is isomorphic to a basic, $\Sec{U}{\pf{\text{im}}\pb{\gamma}F}$ is dualizable. But we see that $\Sec{U}{\pf{\text{im}}\pb{\gamma}F}$ is equivalent to the value of $F$ at the object $$(U\subseteq X)\in\exit{}{X}=(\cate{B}\text{sc}_{/X})\op,$$ and therefore is dualizable by assumption.
    \end{proof}

    \begin{example}\label{exampleconsafterexit}
    	Consider again the stratified space $(X\rightarrow P)\coloneqq(\rnum_{>0}\rightarrow\{0<1\})$ as in \cref{exampleexit}, and let $\C$ be a stable and bicomplete $\infty$-category. It follows from \cref{exodromy} that giving a formally constructible sheaf on $X$ is essentially the same as providing two objects $M$ and $N$ in $\C$, a $\mathbb{Z}$-action on $N$, and a $\mathbb{Z}$-equivariant map $\alpha:M\rightarrow N$, where $M$ is equipped with the trivial action. One may equivalently supply a $\mathbb{Z}$-equivariant object $N$ and a map $\tilde{\alpha}: M\rightarrow N^{h\mathbb{Z}}$, where the target of $\tilde{\alpha}$ denotes the homotopy fixed points.
    \end{example}

    \begin{remark}
        Notice that even though we assumed from the beginning that the coefficients are stable and bicomplete, all the arguments we have discussed work whenever $\Sh{X}{\C}\hookrightarrow\Fun(\cate{U}(X)\op, \C)$ admits a left adjoint and $\C$ respects glueings in the sense of \cite[Definition 5.17]{Haine2020TheHO}. In particular, our proof also recovers the case $\C = \spaces$. A proof of the exodromy equivalence with presentable coefficients but on a much bigger class of stratified spaces can be found in \cite{porta2022topological}. 
    \end{remark}

    \begin{corollary}\label{globsecconstrshcpt}
        Let $Z$ be any compact conically smooth stratified space, and let $F\in\cSh{Z}{\C}$. Then $\Sec{Z}{F}$ is dualizable.
    \end{corollary}

    \begin{proof}
    	By \cref{exodromy}, we know that there exists an essentially unique functor $G : \exit{}{Z}\rightarrow \C^{\text{dual}}$ such that $\pf{\text{im}}\pb{\gamma}G\simeq F$. Therefore, since $\gamma$ is final, global sections of $F$ are equivalent to the limit of $G$. The proof is then concluded by applying \cref{finexit} and observing that, since $\C$ is stable and its monoidal structure is closed, $\C^{\text{dual}}$ is itself stable.
    \end{proof}

    \begin{remark}
    	Let $Z$ be a stratified space such that $\exit{}{Z}$ is a retract in $\infcat$ of a finite $\infty$-category, and assume that the exodromy equivalence holds for constructible sheaves on $Z$. Since $\C^{\text{dual}}$ is idempotent complete, the same argument of \cref{globsecconstrshcpt} shows that for any $F\in\cSh{X}{\C}$, the object $\Sec{Z}{F}$ is dualizable.
    \end{remark}
	
    \section{Verdier duality}
	    This final section is devoted to proving Verdier duality for conically smooth spaces (\cref{verdduality}). For this reason, from now on our $\infty$-categories of coefficients are assumed to be equipped with a closed symmetric monoidal structure. We first introduce the Verdier duality functor, and then recall the definition of Lurie's \textit{covariant Verdier duality}. A crucial observation for the proof stategy that we adopt is that these two functors are closely related. 
	    
	     For any locally compact Hausdorff topological space $X$, we will denote by $\omega_X^{\C}$ the sheaf $\pbp{a}(\mathds{1}_{\C})$, where $a: X\rightarrow\ast$ is the unique map and $\mathds{1}_{\C}$ is the monoidal unit in $\C$. The sheaf $\omega^{\C}_X$ will be called the $\C$-valued \textit{dualizing sheaf} of $X$. We denote the functor
	    	$$\sHom{X}{-}{\omega_X^{\C}} : \Sh{X}{\C}\op\rightarrow\Sh{X}{\C}$$
	    simply by $D_X^{\C}$ and, when $X = \ast$, we will only write $D^{\C}: \C\op\rightarrow\C$. In this case, $D^{\C}$ sends an object $M\in\C$ to its dual $\sHom{\C}{M}{\mathds{1}_{\C}}$. Therefore, $D^{\C}$ gives an equivalence between $\C^{\text{dual}}$ and its opposite.
	    
		Recall that, for any $F\in\Sh{X}{\C}$ and $V\in\cate{U}(X)$, one defines the \textit{compactly supported sections} of $F$ at $V$ by
		\begin{equation*}
			\cSec{V}{F}\coloneqq\varinjlim\limits_{K\subseteq V}\KSec{V}{F}{K}.
		\end{equation*}
		In the colimit above, $K$ ranges through the compact subsets of $V$, and $\KSec{V}{F}{K}$ denotes the fiber of the restriction $\Sec{V}{F}\rightarrow\Sec{V\setminus K}{F}$ (see for example \cite[Definition 5.6]{volpe2021six} and the whole section for a more detailed discussion). The association $F\mapsto\cSec{X}{F}$ gives a left adjoint to the functor $\pbp{a}$. 
		The above construction can be upgraded to a functor
		$$
		\begin{tikzcd}[row sep=tiny]
			\Sh{X}{\C}\arrow[r, "{\mathbb{D}_X^{\C}}"] & \coSh{X}{\C} \\
			F\arrow[r, maps to] & (U\mapsto\cSec{U}{F}).
		\end{tikzcd}
		$$
		In \cite[Theorem 5.5.5.1]{lurie2017higher}, Lurie shows that $\mathbb{D}$ is an equivalence of $\infty$-categories. This equivalence is referred to as \textit{covariant Verdier duality}. Our next lemma explains the relation between covariant Verdier duality and the contravariant functor $D_X^{\C}$.
		
	\begin{lemma}\label{factorizationcovverdcontrverd}
		Let $X$ be any locally compact Hausdorff topological space, and let $\C$ be any stable bicomplete $\infty$-category equipped with a closed symmetic monoidal structure. Then there is a factorization 
		$$
		\begin{tikzcd}
			\Sh{X}{\C}\op\arrow[rr, "{\sHom{X}{-}{\omega_X}}"]\arrow[dr, "\mathbb{D}"', "\simeq"] & & \Sh{X}{\C}. \\
			& \coSh{X}{\C}\op \arrow[ur, "{D^{\C}_{\bullet}}"']
		\end{tikzcd}
		$$
		where $D^{\C}_{\bullet}$ denotes the functor obtained by postcomposing with $D^{\C} : \C\op\rightarrow\C$.
	\end{lemma}

    \begin{proof}
    		Let $j:U\hookrightarrow X$ be any open subset of $X$. Then, for any $F\in\Sh{X}{\C}$, by applying \cite[Corollary 3.26]{volpe2021six}, \cite[Lemma 6.5]{volpe2021six} and \cite[Proposition 6.12]{volpe2021six} we get functorial equivalences
    	\begin{align*}
    		\Sec{U}{\sHom{X}{F}{\omega_X}}&\simeq\Sec{U}{\sHom{U}{\pb{j}F}{\pb{j}\omega_X^{\C}}} \\
    		& \simeq\Sec{U}{\sHom{U}{\pb{j}F}{\omega_U^{\C}}} \\
    		& \simeq\sHom{\C}{\cSec{U}{F}}{\mathds{1}_{\C}}
    	\end{align*}
    	and thus we have the desired factorization.
    \end{proof}

    Our next goal to prove \cref{verdduality} is to show that the covariant Verdier duality functor $\mathbb{D}$ sends constructible sheaves to constructible cosheaves. This is a consequence of constructibility of the dualizing sheaf, that we prove in \cref{dualconstr}. We first compute the stalk at the cone point of the dualizing sheaf of a cone on a compact $C^0$-stratified space.
    
    \begin{lemma}\label{stalkatconeptdualizing}
    	Let $Z$ be a compact $C^0$-stratified space. Denote by $X$ the cone $C(Z)$, and let $x\in C(Z)$ be the cone point. Let $\mathds{1}\in\C$ be the monoidal unit, and let $\mathds{1}_X\in\Sh{X}{\C}$ be the constant sheaf at $\mathds{1}$. Then we have an equivalence
    	$$(\omega_X^{\C})_x\simeq D^{\C}(\KSec{X}{\mathds{1}_X}{\{x\}}).$$
    \end{lemma}

    \begin{proof}
    	We have equivalences
    	\begin{align*}
    		(\omega_X^{\C})_x &= \varinjlim\limits_{x\in U}\Sec{U}{\omega_X^{\C}} \\
    		& \simeq \varinjlim\limits_{0<\epsilon\leq\infty}\Sec{C_{\epsilon}(Z)}{\omega_X^{\C}} \\
    		& \simeq\varinjlim\limits_{0<\epsilon\leq\infty}D^{\C}(\cSec{C_{\epsilon}(Z)}{\mathds{1}_X}).
    	\end{align*}
    	where the second follows from \cref{basisconept} and the third by \cite[Proposition 6.12]{volpe2021six}. Here $C_{\infty}(Z)$ refers to $C(Z)$. 
        Notice that the colimit above is indexed by a weakly contractible $\infty$-category. Therefore it will suffice to show that, for any $\epsilon$, the map $$\KSec{X}{\mathds{1}_X}{\{x\}}\simeq\KSec{C_{\epsilon}(Z)}{\mathds{1}_X}{\{x\}}\rightarrow\cSec{C_{\epsilon}(Z)}{\mathds{1}_X}$$
    	is invertible (see \cite[Remark 5.7]{volpe2021six} for a proof of why the first equivalence holds).
    	
    	First of all, notice that for any $K\subseteq C_{\epsilon}(Z)$ compact containing the cone point, there exists a $T\geq 0$ such that $K\subseteq\overline{C_{T}(Z)}$ (namely, take $T$ to be the maximum in the image of $K$ through the projection $C_{\epsilon}(Z)\rightarrow\mathbb{R}_{\geq 0}$). Hence, by a cofinality argument, we have a commutative triangle
    	$$
    	\begin{tikzcd}
    		\KSec{X}{\mathds{1}_X}{\{x\}}\arrow[d]\arrow[dr] & \\
    		\varinjlim\limits_{0\leq T<\epsilon}\KSec{X}{\mathds{1}_X}{\overline{C_{T}(Z)}} \arrow[r, "\simeq"] &  \cSec{C_{\epsilon}(Z)}{\mathds{1}_X}
    	\end{tikzcd}
    	$$
    	where we fix $\{x\}\coloneqq\overline{C_{T}(Z)}$. 
    	Notice that $0$ is the initial object in the indexing poset of the colimit appearing above. Hence, to conclude our proof it suffices to prove that, for any $T$, the map $$\KSec{X}{\mathds{1}_X}{\{x\}}\rightarrow\KSec{X}{\mathds{1}_X}{\overline{C_{T}(Z)}}$$ is invertible. By definition, this holds if and only if the restriction $$\Sec{X\setminus\{x\}}{\mathds{1}_X}\rightarrow\Sec{X\setminus\overline{C_{T}(Z)}}{\mathds{1}_X}$$ is invertible. But the inclusion $X\setminus\overline{C_{T}(Z)}\hookrightarrow X\setminus\{x\}$ is a homotopy equivalence, and so we may conclude by the homotopy invariance of the shape.
    \end{proof}

	\begin{proposition}\label{dualconstr}
		Let $X$ be any $C^0$-stratified topological space. Then the dualizing sheaf $\omega_X^{\C}$ is constructible.
	\end{proposition}
    \begin{proof}	    
	    
	    We will proceed by induction on the depth of $X$. If $X$ has depth 0, then $X$ is a topological manifold (see \cite[Lemma 2.22]{nocera2021whitney}), and hence, by \cite[Proposition 6.18]{volpe2021six}, $\omega_X^{\C}$ is locally equivalent to $\Sigma^{\text{dim}(X)}\mathds{1}_X$. Now assume that $X$ has finite non-zero depth. Since the question is local on $X$, by \cite[Lemma 2.2.2]{ayala2017local} we may assume that $X = \mathbb{R}^n\times C(Z)$, where $Z$ is a compact $C^0$-stratified space with $\text{depth}(Z)<\text{depth}(X)$.

    	Let $p : \mathbb{R}^n\times C(Z)\rightarrow C(Z)$ be the projection and $b : C(Z)\rightarrow\ast$ the unique map. By \cite[Proposition 6.18]{volpe2021six}, for any sheaf $F$ on $C(Z)$ we have a functorial equivalence $\pbp{p}F\simeq\Sigma^n\pb{p}F$, so it suffices to show that $\pbp{b}_{\C}\s{}=\omega^{\C}_{C(Z)}$ is constructible. Hence we may assume that $X = C(Z)$.     	
    	 
    	Let $x$ be the cone point and $j :  U\hookrightarrow X$ its open complement. Since $x$ is the point at which the depth is maximal, we have $\dpt{U}<\dpt{X}$. Moreover we have an equivalence $\pb{j}\omega_X^{\C}\simeq\omega_U^{\C}$, and so by the inductive hypothesis $\pb{j}\omega_X^{\C}$ is constructible. Thus, for every stratum $T\subseteq X$ which does not contain the cone point, the restriction of $\omega_X$ is locally constant with dualizable stalks. Hence it remains to prove that the stalk of $\omega_X$ at the cone point is dualizable. 
    
        Since the dual of a dualizable object is again dualizable, by \cref{stalkatconeptdualizing} it suffices to show that $\KSec{X}{\mathds{1}_X}{\{x\}}$ is dualizable. By definition, $\KSec{X}{\mathds{1}_X}{\{x\}}$ is the fiber of the restriction $\Sec{X}{\mathds{1}_X}\rightarrow\Sec{X\setminus\{x\}}{\mathds{1}_X}$. Using the homotopy invariance of the shape, we get a commutative square
        $$
        \begin{tikzcd}
        	\Sec{X}{\mathds{1}_X}\arrow[r]\arrow[d, "{\simeq}"] & \Sec{X\setminus\{x\}}{\mathds{1}_X}\arrow[d, "{\simeq}"] \\
        	\mathds{1}\arrow[r] & \Sec{Z}{\mathds{1}_Z}
        \end{tikzcd}
        $$
        where the right vertical arrow is induced by any inclusion $Z\hookrightarrow X$ inducing an homotopy equivalence between $Z$ and $X\setminus\{x\}$, and the left vertical arrow by the inclusion of $x$ in $X$. Since $\C^{\text{dual}}$ is stable, it suffices to show that $\Sec{Z}{\mathds{1}_Z}$ is dualizable. But this follows from \cref{cptloccontrhasdualglobsec}. 
\end{proof}

\begin{corollary}\label{covverdpreservesformconstr}
	Let $X$ be a conically smooth stratified space. Then the covariant Verdier duality functor $\mathbb{D}_{X}^{\C}$ restricts to an equivalence
    \begin{equation}\label{verdfc}
		\mathbb{D}: \fcSh{X}{\C}\simeq\fccoSh{X}{\C}.
	\end{equation}
\end{corollary}

\begin{proof}
	The inverse of the covariant Verdier functor $\mathbb{D}_{X}^{\C}$ is given by $(\mathbb{D}_{X}^{\C\op})\op$ (for example, see the proof of \cite[Theorem 5.10]{volpe2021six}). Therefore, is suffices to show that $\mathbb{D}_{X}^{\C}$ preserves formally constructible objects for any $\C$ stable and bicomplete.
	
	First of all, we prove that if $F\in\Sh{X}{\C}$ is locally constant, then $\mathbb{D}F$ is a formally constructible cosheaf. Since restricting along an open immersion commutes with $\mathbb{D}$ (see \cite[Lemma 6.5]{volpe2021six}), and the property of being formally constructible can be checked on an open cover, it suffices to show that $\mathbb{D}$ sends constant sheaves to formally constructible sheaves. 
	
	Assume $F\simeq\pb{a}M$ where $a : X\rightarrow\ast$ is the unique map. In this case, by \cite[Definition 6.1]{volpe2021six} we have that $\mathbb{D}F\simeq\pbp{a}_{\C\op}M$. Moreover, by \cite[Proposition 6.16]{volpe2021six}, we have that $\pbp{a}_{\C\op}M\simeq\omega_X^{\C\op}\otimes\pb{a}M$. Therefore, by \cref{dualconstr} we get that $\mathbb{D}F$ is constructible.
	
	Assume now that $F$ is any formally constructible sheaf, and let $i : X_{\alpha}\hookrightarrow X$ by the inclusion of a stratum of $X$, with complement $j:U\hookrightarrow X$. We need to show that $\pb{i}_{\C\op}\mathbb{D}F\simeq\mathbb{D}\pbp{i}_{\C}F$ is locally constant. Notice that it suffices to show that $\pbp{i}_{\C}F$ is locally constant. Indeed, since $X_{\alpha}$ is a smooth manifold, it is in particular conically smooth. Therefore, by what we have proven before, if $G$ is any locally constant sheaf on $X_{\alpha}$, $\mathbb{D}G$ must be formally constructible. But $X_{\alpha}$ is unstratified, and hence being formally constructible on $X_{\alpha}$ is equivalent to being locally constant. But by \cref{shriekrestrofformconstrislocconst} we know that $\pbp{i}_{\C}F$ is locally constant, and therefore our proof is concluded.
\end{proof}

\begin{proposition}\label{covverdpreservesconstr}
	Let $X$ be a conically smooth stratified space. Then the covariant Verdier duality functor $\mathbb{D}_{X}^{\C}$ restricts to an equivalence
	\begin{equation}\label{verdc}
		\mathbb{D}: \cSh{X}{\C}\simeq\ccoSh{X}{\C}.
	\end{equation}
\end{proposition}

\begin{proof}
	By \cref{covverdpreservesformconstr}, it suffices to show that, for any $x\in X$ and $F\in\fcSh{X}{\C}$, $F_x\in\C$ is dualizable if and only if $(\mathbb{D}F)_x$ is dualizable, where the latter denotes the costalk of $\mathbb{D}F$ at $x$. 
	
	Let $x : \ast\hookrightarrow X$ be the inclusion of a point $x\in X$. By definition, there are equivalences $(\mathbb{D}F)_x\simeq \pb{x}_{\C\op}(\mathbb{D}F)\simeq\pbp{x}_{\C}F$. Thus, by applying global sections to the localization sequence associated to $i$ and $j$, we obtain a fiber sequence
	$$\pbp{x}F\simeq\Sec{X}{\pf{i}\pbp{i}F}\rightarrow \Sec{X}{F}\rightarrow\Sec{U}{F}$$
	and hence an equivalence $$\KSec{X}{F}{\{x\}}\simeq \pbp{i}F.$$
	Thus, by choosing a conical chart $\rnum^n\times C(Z)$ around $x$ and by applying \cref{stalksconstr}, we get a fiber sequence
	$$(\mathbb{D}F)_x\rightarrow F_x\rightarrow\Sec{(\rnum^n\times C(Z))\setminus(0, \ast)}{F}$$
	where $\ast\in C(Z)$ denotes the cone point. Therefore, arguing as in the proof of \cref{globsecconstrshcpt}, it suffices to show that $\exit{}{(\rnum^n\times C(Z))\setminus(0, \ast)}$ is finite. But by Van Kampen for exit paths, one has a pushout 
	$$
	\begin{tikzcd}
		\exit{}{\rnum^n\setminus\{0\}\times \rnum_{>0}\times Z} \arrow[r] \arrow[d] & \exit{}{\rnum^n\times\rnum_{>0}\times Z} \arrow[d] \\
		\exit{}{\rnum^n\setminus\{0\}\times C(Z)} \arrow[r] & \exit{}{\rnum^n\times C(Z))\setminus(0, \ast)}. \arrow[ul, phantom, "\usebox\pushout" , very near start, color=black]
	\end{tikzcd}
	$$  
	The result then follows by observing that $\text{Exit}$ commutes with products, $\exit{}{Z}$ and $\exit{}{C(Z)}$ are both finite by \cref{finexit} and \cref{exitcone}, and $\exit{}{\rnum^n\setminus\{0\}}\simeq \sing{S^{n-1}}$ is finite.
\end{proof}

\begin{theorem}\label{verdduality}
	Let $X$ be a conically smooth stratified space. Then the restriction to $\cSh{X}{\C}\op$ of the functor $D_X^{\C}$ factors through an equivalence 
	$$
	\begin{tikzcd}
		D_X^{\C}:\cSh{X}{\C}\op\arrow[r, "\simeq"] & \cSh{X}{\C}.
	\end{tikzcd}
	$$
\end{theorem}

\begin{proof}
	By \cref{factorizationcovverdcontrverd} and \cref{covverdpreservesconstr}, we only need to show that 
	$$D^{\C}_{\bullet} : \ccoSh{X}{\C}\op\rightarrow\cSh{X}{\C}$$
	is an equivalence. Denote again by $\text{im}:\text{Bsc}_{/X}\rightarrow\cate{U}(X)$ the functor taking a conically smooth open immersion into $X$ to its image. We see that the diagram
	$$
	\begin{tikzcd}
		\Fun(\exit{}{X}, (\C^{\text{dual}})\op) \arrow[r, "D^{\C}_{\bullet}"] \arrow[d, "\pb{\gamma}", "\simeq"'] & \Fun(\exit{}{X}, \C^{\text{dual}}) \arrow[d, "\pb{\gamma}", "\simeq"'] \\
		\Fun_{\cate{W}}((\text{Bsc}_{/X})\op, (\C^{\text{dual}})\op) \arrow[r, "D^{\C}_{\bullet}"] & \Fun_{\cate{W}}((\text{Bsc}_{/X})\op, \C^{\text{dual}}) \\
		\ccoSh{X}{\C}\op \arrow[u, "\pb{\text{im}}"', "\simeq"] \arrow[r, "D^{\C}_{\bullet}"] & \cSh{X}{\C} \arrow[u, "\pb{\text{im}}"', "\simeq"]
	\end{tikzcd}
	$$
	commutes, since the horizontal arrows are given by postcompositions and the vertical arrows by precompositions. Moreover, since the restriction of $D^{\C}$ induces a duality on $\C^{\text{dual}}$, the upper horizontal arrows are equivalences, and thus we get the desired conclusion.
\end{proof}

\begin{example}
	Let us give an explicit description of what Verdier duality looks like for the stratified space $X$ appearing in \cref{exampleexit}. Let $\C$ be any stable and bicomplete $\infty$-category. For any map $\alpha:M\rightarrow N^{h\mathbb{Z}}$ in $\C$ as in \cref{exampleconsafterexit}, we get a map $\Omega N^{h\mathbb{Z}}\rightarrow\text{fib}(\alpha)$. It follows by Poincar\'e duality for manifolds (see \cite[Proposition 6.18]{volpe2021six}) that there is an equivalence $N^{h\mathbb{Z}}\simeq\Omega N_{h\mathbb{Z}}$. Therefore, we have a map $\Omega^2 N_{h\mathbb{Z}}\rightarrow\text{fib}(\alpha)$. This can be upgraded to a functor
	$$
	\begin{tikzcd}[row sep=tiny]
		\Fun(B\mathbb{Z}^{\triangleleft},\C)\arrow[r] & \Fun(B\mathbb{Z}^{\triangleright},\C) \\
		(\alpha:M\rightarrow N^{h\mathbb{Z}})\arrow[r, maps to] & (\tilde{\alpha}:\Omega^2 N_{h\mathbb{Z}}\rightarrow\text{fib}(\alpha)),
	\end{tikzcd}
    $$
    which is easily seen to be an equivalence. We invite the interested reader to work out the details to show that the functor given above coincides with (\ref{verdfc}), after applying the exodromy equivalence.
\end{example}


\begin{remark}\label{>AMGR}
	In \cite[Example 1.10.8]{ayala2019stratified}, the authors propose a strategy to prove Verdier duality. However, they do not provide proofs for some of the major steps in their outline. We here specify what are the main missing points in \cite[Example 1.10.8]{ayala2019stratified}. Let $X\rightarrow P$ be any stratified topological space. First of all, in \cite[Example 1.10.8]{ayala2019stratified} there is no explanation of why the stratification  on $\Sh{X}{\C}$ restricts to a stratification on $\fcSh{X}{\C}$. We verify this for conically smooth stratified spaces in the proof of \cref{shriekrestrofformconstrislocconst}. Secondly, in \cite[Example 1.10.8]{ayala2019stratified} the authors claim without proof that, if $\omega_X$ is formally constructible, then the covariant Verdier duality functor preserves formally constructible objects. We prove this claim in \cref{covverdpreservesformconstr}. The authors also do not explain for which kind of stratified topological spaces one should expect the dualizing sheaf to be formally constructible. We show that this is the case for $C^0$-stratified spaces in \cref{dualconstr}.
\end{remark}

\begin{remark}
    Notice that the equivalences (\ref{verdfc}) and (\ref{verdc}) are already interesting on their own, because they imply that for any stratified map $f:X\rightarrow Y$, $\pf{f}^{\C}$ or $\pb{f}_{\C}$ preserves (formal) constructibility if and only if $\pfp{f}^{\C}$ or $\pbp{f}_{\C}$ does. In particular, we see that $\pbp{f}_{\C}$ always preserves (formally) constructible sheaves.
\end{remark}

\begin{remark}
	Any $\mu$-stratification of an analytic manifold in the sense of \cite{kashiwara1990sheaves} satisfies the Whitney conditions, and hence by \cite{nocera2021whitney} defines a conically smooth structure. Thus, \cref{verdduality} recovers and generalizes the duality on constructible sheaves on analytic manifolds as defined in \cite{kashiwara1990sheaves} (i.e. sheaves which are constructible in our sense with respect to \textit{some} $\mu$-stratification).
\end{remark}


\appendix

\section{The shape of a proper locally contractible $\infty$-topos}

In this appendix, we present a couple of topos-theoretic results that provide an elegant argument to conclude the last step in the proof of \cref{dualconstr}. While the content of this appendix is not fundamental for our primary purposes, we have decided to include it because we believe it is interesting on its own. More precisely, in this appendix we prove that the \textit{shape} of any \textit{proper} and \textit{locally contractible} $\infty$-topos is a compact $\infty$-groupoid.

We start by recalling the definition of the shape of a locally contractible $\infty$-topos. For a more general and detailed discussion about shape and locally contractible geometric morphisms, see \cite[Section 3]{volpe2021six}.

\begin{definition}
	Let $\X$ be an $\infty$-topos, and let $a:\X\rightarrow\spaces$ be the unique geometric morphism. We say that $\X$ is \textit{locally contractible} if $\pb{a}:\spaces\rightarrow\X$ admits a left adjoint, denoted by $\pfs{a}:\X\rightarrow\spaces$.
	
	If $\X$ is any locally contractible $\infty$-topos, we define the \textit{shape} of $\X$, denoted by $\Pi_{\infty}(\X)$, as the $\infty$-groupoid $\pfs{a}(1_{\X})$, where $1_{\X}$ denotes the terminal object of $\X$.
\end{definition}

We now show that sheaf topoi associated to $C^0$-stratified spaces are locally contractible. We need the following preliminary lemma.

\begin{lemma}\label{C0ishypercomplete}
	Let $X$ be a $C^0$-stratified space. Then the $\infty$-topos $\Sh{X}{\spaces}$ is hypercomplete.
\end{lemma}

\begin{proof}
	By \cite[Lemma 2.2.2]{ayala2017local}, $X$ admits an open cover given by its open subsets isomorphic as stratified spaces to one of the type $\mathbb{R}^n\times C(Z)$, where $Z$ is a compact $C^0$-stratified space. Therefore, $X$ is locally paracompact and of finite covering dimension. By \cite[Theorem 7.2.3.6]{lurie2009higher}, the covering dimension of a paracompact space agrees with its \textit{homotopy dimension} (see \cite[Definition 7.2.1.1]{lurie2009higher}). Moreover, by \cite[Corollary 7.2.1.12]{lurie2009higher} any $\infty$-topos which is locally of finite homotopy dimension is hypercomplete. Therefore, we conclude that $\Sh{X}{\spaces}$ is hypercomplete.
\end{proof}

\begin{corollary}
	Let $X$ be a $C^0$-stratified space. Then the $\infty$-topos $\Sh{X}{\spaces}$ is locally contractible. Moreover, we have an equivalence of $\infty$-groupoids $\Pi_{\infty}(\Sh{X}{\spaces})\simeq\sing{X}$.
\end{corollary}

\begin{proof}
	Since by \cite[Lemma 2.2.2]{ayala2017local} the topological space $X$ is locally contractible, the result follows from \cref{C0ishypercomplete} and \cite[Corollary 3.19]{volpe2021six}.
\end{proof}

\begin{definition}\label{defpropertopos}
	Let $\X$ be an $\infty$-topos, and let $a:\X\rightarrow\spaces$ be the unique geometric morphism. We say that $\X$ is \textit{proper} if $\pf{a}:\X\rightarrow\spaces$ preserves filtered colimits.
\end{definition}

\begin{remark}
	It would be very natural to define an $\infty$-topos to be proper by requiring the unique geometric morphism $a:\X\rightarrow\spaces$ to be proper in the sense of \cite[Definition 7.3.1.4]{lurie2009higher}. This alternative definition is proven to be equivalent to \cref{defpropertopos} in \cite{martini2023proper}.
\end{remark}

\begin{proposition}\label{cptshapepropertopos}
	Let $\X$ be a proper and locally contractible $\infty$-topos. Then $\Pi_{\infty}(\X)$ is a compact object in $\spaces$.
\end{proposition}

\begin{proof}
	Let $a:\X\rightarrow\spaces$ be the unique geometric morphism. Observe that $\pf{a}:\X\rightarrow\spaces$ is corepresented by the terminal object $1_{\X}$. Since $\X$ is assumed to be proper, we see that $1_{\X}$ must be a compact object in $\X$. Therefore, to conclude the proof it suffices to show that $\pfs{a}:\X\rightarrow\spaces$ preserves compact objects. But this is clear because its right adjoint $\pb{a}$ preserves (filtered) colimits.
\end{proof}

\begin{corollary}\label{cptloccontrhasdualglobsec}
	Let $X$ be any compact Hausdorff topological space, and assume that $\Sh{X}{\spaces}$ is locally contractible. Let $\C$ be any stable bicomplete $\infty$-category equipped with a closed symmetric monoidal structure. Let $M\in\C$ be any dualizable object, and denote by $M_X$ the constant sheaf at $M$. Then $\Sec{X}{M_X}$ is dualizable.
\end{corollary}

\begin{proof}
	Let $a:X\rightarrow\ast$ be the unique map. Recall that, by \cite[Corollary 5.16]{volpe2021six}, we have an equivalence $\Sh{X}{\C}\simeq\Sh{X}{\spaces}\otimes\C$, where the $\otimes$ denotes Lurie's tensor product of cocomplete $\infty$-categories. Since $X$ is locally contractible, combining \cite[Corollary 5.16]{volpe2021six} and \cite[Corollary 5.20]{volpe2021six}, we see that $\pb{a}_{\C}:\C\rightarrow\Sh{X}{\C}$ admits a left adjoint $\pfs{a}^{\C}$ obtained by tensoring with $\C$ the cocontinuous functor $\pfs{a}:\Sh{X}{\spaces}\rightarrow\spaces$. In particular, if we denote by $\mathds{1}_X$ the constant sheaf at the monoidal unit $\mathds{1}\in\C$, we get that $$\pfs{a}^{\C}(\mathds{1}_X)\simeq\varinjlim\limits_{\Pi_{\infty}(X)}\mathds{1}.$$
	Here $\Pi_{\infty}(X)$ denotes the shape of the locally contractible $\infty$-topos $\Sh{X}{\spaces}$. Moreover, it follows from the dual version of the smooth projection formula (see \cite[Corollary 3.26]{volpe2021six}) that there is an equivalence $\Sec{X}{M_X}\simeq\sHom{\C}{\pfs{a}^{\C}(\mathds{1}_X)}{M}$. Hence we have $$\Sec{X}{M_X}\simeq\varprojlim\limits_{\Pi_{\infty}(X)}M.$$
	Since $\C^{\text{dual}}$ is an idempotent complete stable $\infty$-category, we can conclude by \cref{cptshapepropertopos}.
\end{proof}

\newpage
\nocite{*}
\bibliographystyle{alpha}
\bibliography{dualizing}

\newcommand{\etalchar}[1]{$^{#1}$}
\begin{thebibliography}{AMGR19}

\bibitem[AFR18]{ayala2018stratified}
David Ayala, John Francis, and Nick Rozenblyum.
\newblock A stratified homotopy hypothesis.
\newblock {\em Journal of the European Mathematical Society}, 21(4):1071--1178,
  2018.

\bibitem[AFT17]{ayala2017local}
David Ayala, John Francis, and Hiro~Lee Tanaka.
\newblock Local structures on stratified spaces.
\newblock {\em Advances in Mathematics}, 307:903--1028, 2017.

\bibitem[AK10]{arone2010functoriality}
Gregory Arone and Marja Kankaanrinta.
\newblock On the functoriality of the blow-up construction.
\newblock {\em Bulletin of the Belgian Mathematical Society-Simon Stevin},
  17(5):821--832, 2010.

\bibitem[AMGR19]{ayala2019stratified}
David Ayala, Aaron Mazel-Gee, and Nick Rozenblyum.
\newblock Stratified noncommutative geometry.
\newblock {\em arXiv preprint arXiv:1910.14602}, 2019.

\bibitem[Aok20]{aoki2020tensor}
Ko~Aoki.
\newblock Tensor triangular geometry of filtered objects and sheaves.
\newblock {\em arXiv preprint arXiv:2001.00319}, 2020.

\bibitem[Aok22]{aoki2022posets}
Ko~Aoki.
\newblock Posets for which verdier duality holds.
\newblock {\em arXiv preprint arXiv:2202.05312}, 2022.

\bibitem[BBDG18]{beilinson2018faisceaux}
Alexander Beilinson, Joseph Bernstein, Pierre Deligne, and Ofer Gabber.
\newblock {\em Faisceaux pervers}.
\newblock Soci{\'e}t{\'e} math{\'e}matique de France, 2018.

\bibitem[BGH18]{barwick2018exodromy}
Clark Barwick, Saul Glasman, and Peter Haine.
\newblock Exodromy.
\newblock {\em arXiv preprint arXiv:1807.03281}, 2018.

\bibitem[Bj{\"o}99]{bjorner1999oriented}
Anders Bj{\"o}rner.
\newblock {\em Oriented matroids}.
\newblock Number~46. Cambridge University Press, 1999.

\bibitem[CDH{\etalchar{+}}20]{calmes2020hermitian}
Baptiste Calm{\'e}s, Emanuele Dotto, Yonatan Harpaz, Fabian Hebestreit, Markus
  Land, Kristian Moi, Denis Nardin, Thomas Nikolaus, and Wolfgang Steimle.
\newblock Hermitian {K}-theory for stable $\infty$-categories {I}: Foundations.
\newblock {\em preprint}, 2020.

\bibitem[Cis19]{cisinski2019higher}
Denis-Charles Cisinski.
\newblock {\em Higher categories and homotopical algebra}, volume 180.
\newblock Cambridge University Press, 2019.

\bibitem[Cur18]{curry2018dualities}
Justin~Michael Curry.
\newblock Dualities between cellular sheaves and cosheaves.
\newblock {\em Journal of Pure and Applied Algebra}, 222(4):966--993, 2018.

\bibitem[DW21]{douteau2021homotopy}
Sylvain Douteau and Lukas Waas.
\newblock From homotopy links to stratified homotopy theories.
\newblock {\em arXiv preprint arXiv:2112.02394}, 2021.

\bibitem[Gor21]{goresky2021lecture}
Mark Goresky.
\newblock Lecture notes on sheaves and perverse sheaves.
\newblock {\em arXiv preprint arXiv:2105.12045}, 2021.

\bibitem[Hai18]{haine2018homotopy}
Peter~J Haine.
\newblock On the homotopy theory of stratified spaces.
\newblock {\em arXiv preprint arXiv:1811.01119}, 2018.

\bibitem[Hai21]{haine2021nonabelian}
Peter~J. Haine.
\newblock From nonabelian basechange to basechange with coefficients.
\newblock {\em arXiv preprint arXiv:2108.03545}, 2021.

\bibitem[Hen]{314697}
Simon Henry.
\newblock Why are finite cell complexes also finite as infinity-categories?
\newblock MathOverflow answer.
\newblock \url{https://mathoverflow.net/q/314697}.

\bibitem[HPT20]{Haine2020TheHO}
Peter~J. Haine, Mauro Porta, and Jean-Baptiste Teyssier.
\newblock The homotopy-invariance of constructible sheaves of spaces.
\newblock {\em arXiv: Algebraic Topology}, 2020.

\bibitem[Kas84]{kashiwara1984riemann}
Masaki Kashiwara.
\newblock The {R}iemann-{H}ilbert problem for holonomic systems.
\newblock {\em Publications of the Research Institute for Mathematical
  Sciences}, 20(2):319--365, 1984.

\bibitem[KS90]{kashiwara1990sheaves}
Masaki Kashiwara and Pierre Schapira.
\newblock {\em Sheaves on {M}anifolds: {W}ith a {S}hort {H}istory. Les
  d{\'e}buts de la th{\'e}orie des faisceaux. By {C}hristian {H}ouzel}, volume
  292.
\newblock Springer Science \& Business Media, 1990.

\bibitem[Lur09]{lurie2009higher}
Jacob Lurie.
\newblock {\em Higher topos theory}.
\newblock Princeton University Press, 2009.
\newblock \url{https://www.math.ias.edu/~lurie/papers/HTT.pdf}.

\bibitem[Lur17]{lurie2017higher}
Jacob Lurie.
\newblock {\em Higher algebra}.
\newblock 2017.
\newblock \url{https://www.math.ias.edu/~lurie/papers/HA.pdf}.

\bibitem[MW23]{martini2023proper}
Louis Martini and Sebastian Wolf.
\newblock Proper morphisms of $\infty $-topoi.
\newblock {\em arXiv preprint arXiv:2311.08051}, 2023.

\bibitem[NV21]{nocera2021whitney}
Guglielmo Nocera and Marco Volpe.
\newblock Whitney stratifications are conically smooth.
\newblock {\em arXiv preprint arXiv:2105.09243}, 2021.

\bibitem[PT22]{porta2022topological}
Mauro Porta and Jean-Baptiste Teyssier.
\newblock Topological exodromy with coefficients.
\newblock {\em arXiv preprint arXiv:2211.05004}, 2022.

\bibitem[Qui88]{quinn1988homotopically}
Frank Quinn.
\newblock Homotopically stratified sets.
\newblock {\em Journal of the American Mathematical Society}, pages 441--499,
  1988.

\bibitem[Rez01]{rezk2001model}
Charles Rezk.
\newblock A model for the homotopy theory of homotopy theory.
\newblock {\em Transactions of the American Mathematical Society},
  353(3):973--1007, 2001.

\bibitem[SW20]{schurmann2020witt}
J{\"o}rg Sch{\"u}rmann and Jonathan Woolf.
\newblock Witt groups of abelian categories and perverse sheaves.
\newblock {\em Annals of K-Theory}, 4(4):621--670, 2020.

\bibitem[Tre09]{treumann2009exit}
David Treumann.
\newblock Exit paths and constructible stacks.
\newblock {\em Compositio Mathematica}, 145(6):1504--1532, 2009.

\bibitem[Vol21]{volpe2021six}
Marco Volpe.
\newblock The six operations in topology.
\newblock {\em arXiv preprint arXiv:2110.10212}, 2021.

\bibitem[Vol23]{volpe2023finiteness}
Marco Volpe.
\newblock Finiteness and finite domination in stratified homotopy theory.
\newblock {\em To appear}, 2023.

\bibitem[Wal65]{10.2307/1970382}
C.~T.~C. Wall.
\newblock Finiteness conditions for {CW}-complexes.
\newblock {\em Annals of Mathematics}, 81(1):56--69, 1965.

\end{thebibliography}
\end{document}